\documentclass[a4paper,11pt,leqno]{article}
\usepackage{geometry}
\usepackage{tikz,tikz-cd,amsmath,amsfonts,amssymb,graphicx}
\usetikzlibrary{matrix,arrows,decorations.pathmorphing}
\usepackage[latin1]{inputenc}
\usepackage{amsthm}
\usepackage[autostyle]{csquotes} 
\usepackage[english]{babel}
\usepackage{enumerate}
\usepackage[mathscr]{euscript}
\usepackage{mathtools}
\usepackage{hyperref}
\usepackage{url}
\usepackage{scalerel}
\usepackage{dsfont}
\usepackage{stackengine,wasysym}

\usepackage[noabbrev,nameinlink,capitalize]{cleveref} 

\numberwithin{equation}{section}
\theoremstyle{definition}

\newtheorem{theorem}[equation]{Theorem}
\newtheorem{corollary}[equation]{Corollary}
\newtheorem{definition}[equation]{Definition}
\newtheorem{lemma}[equation]{Lemma}
\newtheorem{example}[equation]{Example}
\newtheorem{proposition}[equation]{Proposition}
\newtheorem{remark}[equation]{Remark}
\newtheorem{construction}[equation]{Construction}
\newtheorem{notation}[equation]{Notation}

\makeatletter
\newcommand{\colim@}[2]{%
	\vtop{\m@th\ialign{##\cr
			\hfil$#1\Operator@font colim$\hfil\cr
			\noalign{\nointerlineskip\kern1.5\ex@}#2\cr
			\noalign{\nointerlineskip\kern-\ex@}\cr}}%
}
\newcommand{\colim}{%
	\mathop{\mathpalette\colim@{\rightarrowfill@\scriptscriptstyle}}\nmlimits@
}
\renewcommand{\varprojlim}{%
	\mathop{\mathpalette\varlim@{\leftarrowfill@\scriptscriptstyle}}\nmlimits@
}
\renewcommand{\varinjlim}{%
	\mathop{\mathpalette\varlim@{\rightarrowfill@\scriptscriptstyle}}\nmlimits@
}
\makeatother
\newcommand{\cate}[1]{\mathscr{#1}}

\newcommand{\Sh}[2]{\cate{S}\text{hv}({#1}; {#2})}
\newcommand{\PSh}[1]{\cate{PS}\text{hv}({#1})}

\newcommand{\Shsp}[1]{\cate{S}\text{hv}({#1})}

\newcommand{\sHom}[3]{\text{\underline{Hom}}_{\scalebox{1}{$\scriptscriptstyle #1$}}(#2, #3)}
\newcommand{\Hom}[3]{\text{Hom}_{\scalebox{1}{$\scriptscriptstyle #1$}}(#2, #3)}

\newcommand{\pf}[1]{#1_{\ast}}
\newcommand{\pb}[1]{#1^{\ast}}
\newcommand{\pfp}[1]{#1_{!}}

\newcommand{\pfs}[1]{#1_{\sharp}}

\DeclareMathOperator{\Fun}{Fun}
\DeclareMathOperator{\Pro}{Pro}
\DeclareMathOperator{\opposite}{op}
\DeclareMathOperator{\lex}{lex}

\newcommand{\op}{^{\opposite}}
\newcommand{\hyp}{^{\text{hyp}}}
\newcommand{\C}{\cate{C}}
\newcommand{\D}{\cate{D}}

\newcommand{\X}{\cate{X}}
\newcommand{\Y}{\cate{Y}}
\newcommand{\Top}{\cate{T}\mathrm{op}}
\newcommand{\Cat}{\cate{C}\mathrm{at}}

\newcommand{\Map}{\mathrm{Map}}
\newcommand{\Z}{\mathcal{Z}}

\newcommand{\An}{\mathrm{An}}

\newcommand{\sh}[1]{\mathrm{sh}(#1)}

\newcommand{\intFun}{\underline{\mathrm{Fun}}}

\providecommand{\infcat}{\Cat_{\scalebox{1}{$\scriptscriptstyle \infty$}}}

\newsavebox{\pullback}
\sbox\pullback{%
	\begin{tikzpicture}%
		\draw (0,0) -- (1ex,0ex);%
		\draw (1ex,0ex) -- (1ex,1ex);%
\end{tikzpicture}}
\newsavebox{\pushout}
\sbox\pushout{%
	\begin{tikzpicture}%
		\draw (0ex,0ex) -- (0ex,1ex);%
		\draw (0ex,1ex) -- (1ex,1ex);%
\end{tikzpicture}}
\tikzset{%
	symbol/.style={%
		draw=none,
		every to/.append style={%
			edge node={node [sloped, allow upside down, auto=false]{$#1$}}}
	}
}

{}

\title{Approximate Fibrations in Higher Topos Theory}
\author{Christian Kremer and Marco Volpe}

\begin{document}

\maketitle

\begin{abstract}
    The goal of this paper is to put the theory of approximate fibrations into the framework of higher topos theory. 
    We define the notion of an approximate fibration for a general geometric morphism of $\infty$-topoi, give several characterizations in terms of shape theory and compare it to the original definition for maps of topological spaces of Coram and Duvall. Furthermore, we revisit the notion of cell-like maps between topoi, and generalize Lurie's shape-theoretic characterization by giving a purely topos-theoretical proof.
\end{abstract}

\section{Introduction}

Approximate fibrations, introduced by Coram and Duvall in \cite{CoramDuvall}, are a class of maps between topological spaces enlarging that of proper Hurewicz fibrations. Informally, these are continuous functions satisfying a more relaxed homotopy lifting property, in the sense that lifts only exist up to an appropriate notion of closeness. We recall the precise definition in \cref{sec:approximate_fibrations}. They admit a surprisingly robust bundle theory: for instance, approximate fibrations between manifolds satisfy an $h$-principle \cite{HughesTaylorWilliams}, and those with manifold total space can be studied via surgery theory \cite{ChapmanApproximation}.

Approximate fibrations appear naturally in geometric topology. A key example comes from neighborhood theory. A theorem of Hughes-Taylor-Weinberger-Williams \cite{HTW} identifies germs of neighborhoods of locally flat submanifolds of high-dimensional manifolds with approximate fibrations having manifold total space and spherical homotopy fiber. Approximate fibrations have also been used to recast Farrell's fibering theorem \cite{FarrellFibering}, which gives a $K$-theoretic obstruction to deforming a map $p\colon M^d \to S^1$ ($d\geq 6$) into a fiber bundle projection with manifold fiber. In fact, such a map is homotopic to a fiber bundle projection precisely when it is homotopic to an approximate fibration. More generally, while direct extensions of Farrell's theorem to arbitrary bases are difficult, one can still find complete $K$-theoretic obstructions to deforming a map to an aspherical manifold into an approximate fibration, for example when the fundamental group is word-hyperbolic \cite{FarrellLueckSteimle}.

A particularly notable subclass of approximate fibrations is given by cell-like maps. These are proper maps whose fibers have contractible shape \cite{lacher1969cell}. Cell-like maps are central to manifold topology. Chapman \cite{chapmantorsion} and West \cite{west1977mapping} used them to prove that every compact topological manifold has a preferred simple homotopy type, well defined up to homeomorphism. They are indispensable tools in the proofs of the generalized Sch\"onflies theorem \cite{Brown1960Schoenflies}, Cannon's double suspension theorem \cite{cannon1979shrinking}, and in the study of manifold factors. The guiding principle, initiated by Siebenmann \cite{SiebenmannCellular}, is that under dimension restrictions, cell-like maps between manifolds can be approximated by homeomorphisms. Finally, cell-like maps and decomposition theory played a decisive role in Freedman's work on the 4-dimensional Poincar\'e conjecture \cite{Freedman,DiscEmbedding}.

These examples suggest that many deep questions in geometric topology can be profitably reformulated in terms of cell-like maps and approximate fibrations. In \cite[Proposition 7.3.6.6]{lurie2009higher}, Lurie gave a topos-theoretic characterization of cell-like maps, recasting their definition in the modern language of $\infty$-topoi. In \cref{sec:recall_topoi} we recall the main notions and facts from higher topos theory that we rely on in this paper, and to formulate the main results in this introduction. The topos theoretic perspective on cell-like maps opens the possibility of re-deriving classical results, such as those of Chapman and West, in a purely categorical framework. This is carried out in an ongoing project of the second named author in collaboration with Maxime Ramzi and Sebastian Wolf, generalizing results previously announced in \cite{NikolausYoutube2023}.

The first contribution of this paper is to generalize Lurie's theorem to a broader setting, including a purely topos-theoretic proof (and generalization) of Lacher's characterization of cell-like maps in terms of hereditary shape equivalences \cite[Theorem 1.2]{lacher1969cell}.

\begin{theorem}\label{introthm:cell-like}
    Let $p:\X\rightarrow\Y$ be a proper geometric morphism. Assume that locally constant objects in $\X$ and $\Y$ are hypercomplete, and that the hypercompletion of $\Y$ has enough points. Then the following are equivalent.
    \begin{enumerate}[(1)]
        \item $p$ is a cell-like geometric morphism, i.e. $\pb{p}$ is fully faithful;
        \item $p$ is a hereditary shape equivalence, i.e. for each $U\in\Y$, the geometric morphism $p_{/U}:\X_{/U}\rightarrow\Y_{/U}$ is a shape equivalence;
        \item for each point $y:\An\rightarrow\Y$, the fiber of $p$ at $y$ has trivial shape.
    \end{enumerate}
\end{theorem}

After passing to the associated sheaf $\infty$-topoi, the third condition in \cref{introthm:cell-like} corresponds to Lacher's original definition of a cell-like mapping of topological spaces. 
The assumption on locally constant objects is satisfied, for example, by topoi of the form $\Sh{X}{\An}$, where $X$ is either an ANR or a paracompact space of finite covering dimension. For any topological space $X$, the hypercompletion of $\Sh{X}{\An}$ has enough points (see \cref{rmk:Xhypenoughpoints}).

The second and more novel contribution is to provide a topos-theoretic characterization of approximate fibrations, entirely parallel to Lurie's characterization of cell-like maps. We make use of the framework of \textit{internal higher categories} developed by Martini-Wolf \cite{martini2025presentabilitytopoiinternalhigher,Martini2023a} to introduce the concept of an approximate fibration in the setting of geometric morphisms of $\infty$-topoi. A proper geometric morphism $f \colon \X \to \Y$ is called an \emph{approximate fibration} if, for every $U \in \Y$, the functor $\pf{{f_{/U}}} \pb{{f_{/U}}} \colon \Y_{/U} \longrightarrow \Y_{/U}$ is corepresentable, subject to additional compatibilities detailed in \cref{subsec:approximate_fibrations}. A reader familiar with shape theory may view the definition of approximate fibrations as a rigorous way to formulate the requirement that the shape of $\X$ be constant, in a manner parametrized by $\Y$. We provide a more intuitive reformulation of approximate fibrations in terms of shape theory, under the assumption that $\X$ and $\Y$ are locally contractible $\infty$-topoi.

\begin{theorem}
    Let $f:\X\rightarrow\Y$ be a proper geometric morphism. Assume that $\X$ and $\Y$ are locally contractible and that the hypercompletion of $\Y$ has enough points. Then the following are equivalent.
    \begin{enumerate}[(1)]
        \item $f$ is an approximate fibration;
        \item $f$ is a shape fibration, i.e. for each $U\in\Y$, the map $\sh{\X_{/\pb{f}U}}\rightarrow\sh{\Y_{U}}\times_{\sh{\Y}}\sh{\X}$ is an equivalence;
        \item $f$ is a shape quasi-fibration, i.e. for each point $y:\An\rightarrow\Y$, the sequence $\sh{\X_y}\rightarrow\sh{\X}\rightarrow\sh{\Y}$ is a fiber sequence.
    \end{enumerate}
\end{theorem}

The condition of being locally contractible is satisfied, for example, by topoi of the form $\Sh{X}{\An}$, where $X$ is a CW-complex or an ANR. The concepts of internally compact objects, internally filtered colimits and the characterization of proper maps in terms of the latter given in \cite{Martini2023a} play a fundamental role in most of our arguments. Combining our characterization with earlier theorems of \cite{HughesTaylorWilliams}, we find that our newly introduced concept of an approximate fibration encompasses the classical one.

\begin{corollary}
    Let $f:X\rightarrow Y$ be a proper map of ANRs. Then $f$ is an approximate fibration if and only if the associated geometric morphism $f:\Sh{X}{\An}\rightarrow\Sh{Y}{\An}$ is an approximate fibration. 
\end{corollary}

Understanding approximate fibrations in these terms is a step toward importing techniques from higher topos theory into manifold topology. We hope this will ultimately clarify how topos theory interacts with classical obstruction theories (such as \cite{FarrellFibering, FarrellLueckSteimle}), and will provide new perspectives on the topology of manifolds.

\subsection*{Acknowledgements}

We thank Shmuel Weinberger for introducing us to the beautiful theory of approximate fibrations, and Emma Brink for helpful comments. We thank an anonymous referee for helpful suggestions. MV thanks Sebastian Wolf for answering questions about internal higher category theory.
CK is supported by the Max Planck Institute for Mathematics in Bonn. MV is supported by the University of Regensburg.

\section{Recollections on topoi and internal category theory}\label{sec:recall_topoi}

We recall the language of $\infty$-topoi and internal higher category theory used throughout this article. We mostly follow from the presentations given in \cite{lurie2009higher, martini2025presentabilitytopoiinternalhigher, Martini2023a}. We also use this section to introduce the notations we use in the rest of the article.

\subsection{Topos theory}

For any $\infty$-category $I$, we write $\PSh{I}$ for the $\infty$-category of presheaves of animae on $I$. An $\infty$-category $\X$ is called an \textit{$\infty$-topos} if there exists a small $\infty$-category $I$ such that $\X$ is a left exact accessible Bousfield localization of $\PSh{I}$. Other useful characterizations are provided by the so called \textit{Giraud's axioms} (see \cite[Section 6.1]{lurie2009higher}).  A \textit{geometric morphism} of topoi $f \colon \X \rightarrow \Y$ is the datum of a right adjoint functor $f_* \colon \X \rightarrow \Y$ whose left adjoint is left exact (i.e. preserves finite limits). We write $\Top$ for the category of topoi and geometric morphisms.

For any topological space $X$, the $\infty$-category $\Sh{X}{\An}$ of sheaves of animae on $X$ is an $\infty$-topos. This example serves as the foundation for most of the applications of interest in this paper. Note that a continuous map $f \colon X \rightarrow Y$ induces a functor $f_*\colon \Sh{X}{\An} \rightarrow \Sh{Y}{\An}$. This is usually referred to as the \textit{pushforward along $f$}. The functor $\pf{f}$ admits a left exact left adjoint $\pb{f}$, called the \textit{pullback along $f$}, and hence it is part of a geometric morphism.

Every topos has a final object, which we usually denote as $1 \in \Y$. The functor corepresented by $1\in\Y$ admits a left exact left adjoint, given as the unique colimit-preserving extension of the functor sending the point $1 \in \An$ to $1 \in \Y$. There is a unique colimit-preserving left exact functor $\An \rightarrow \Y$, so $\An$ is the terminal object in $\Top$. 

For an object $U$ in a topos $\Y$, the forgetful map $(\pi_U)_\sharp \colon \Y_{/U} \rightarrow \Y$ admits a right adjoint which itself admits a right adjoint. The geometric morphism presented by the double-right adjoint of $(\pi_U)_\sharp$ will be denoted by $\pi_U \colon \Y_{/U} \rightarrow \Y$. An \textit{\'etale geometric morphism} is a geometric morphism equivalent in $\Top$ to a morphism of the type $\pi_U \colon \Y_{/U} \rightarrow \Y$ for some topos $\Y$ and $U\in \Y$.

A map $f \colon U \rightarrow V$ in a topos $\Y$ is an \textit{effective epimorphism} if the augmented simplicial diagram $(U\times_V \dots \times_V U)_{p \in \Delta\op} \rightarrow V$ is a colimit diagram. An \'etale morphism is an \textit{\'etale cover} if it is equivalent to $\pi_U \colon \Y_{/U} \rightarrow \Y$ for some effective epimorphism $U \rightarrow 1$.

An object $F \in \Y$ is \textit{constant} if it lies in the essential image of $p^* \colon \An \rightarrow \Y$, where $p$ is the unique geometric morphism to $\An$. An object $F \in \Y$ is \textit{locally constant} if there exists a collection of objects $V_i \in \Y$ such that $\coprod_i V_i \rightarrow 1$, $i \in I$ is an effective epimorphism, and $\pi_{V_i}^*(F) \in \Y_{/V_i}$ is constant for each $i\in I$. A locally constant object is said to have \textit{compact values} if there exists an \'etale cover of $1$ as above such that each restriction $\pi_{V_i}^*(F) \in \Y_{/V_i}$ is the pullback of a compact object in $\An$ along the unique geometric morphism $ \Y_{/V_i}\rightarrow\An$. We write $\Y^{lc}$ for the full subcategory of $\Y$ spanned by the locally constant objects.

\subsection{Points and hypercompletions}

An object $U \in \Y$ is \textit{$n$-truncated} if $\Map(V,U) \in \An$ is an $n$-truncated space for each $V \in Y$. A morphism $f \colon U \rightarrow V$ is $\infty$-connective if for each $n \geq 0$ and each $n$-truncated object $W \in \Y_{/V}$ the map $\Map_{\Y_{/V}}(V,W) \rightarrow \Map_{\Y_{/V}}(U,W)$ is an equivalence. If $f \colon \X \rightarrow \Y$ is a geometric morphism, then $f^* \colon \Y \rightarrow \X$ sends $\infty$-connective morphisms to $\infty$-connective morphisms \cite[Prop. 6.5.1.16.]{lurie2009higher}.

An object $U$ in an arbitrary $\infty$-topos is \textit{hypercomplete} if it is local with respect to $\infty$-connective morphisms - for each $\infty$-connective morphism $f \colon V \rightarrow W$ in $\X$, the map
\[ f^* \colon \Map(W,U) \rightarrow \Map(V,U) \]
is an equivalence. An $\infty$-topos is said to be \textit{hypercomplete} if each of its objects is hypercomplete. Since $\pb{f}$ preserves $\infty$-connective morphisms, it follows that $\pf{f}$ preserves hypercomplete objects.

The inclusion of the full subcategory on the hypercomplete topoi has a right adjoint $(-)^{\hyp}$ called \textit{hypercompletion}, whose adjunction counit $j \colon \Y^{\hyp} \rightarrow \Y$ is such that $j_*$ is equivalent to the inclusion of the subcategory spanned by the hypercomplete objects.

A \textit{point} of an $\infty$-topos $\X$ is a geometric morphism $\An \xrightarrow{x} \X$. Since the $\infty$-topos $\An$ is hypercomplete, we get an equivalence in $\An$
\[ \Map_{\Top}(\An, \X) \simeq \Map_{\Top}(\An,\X^{\hyp}) \]
of points in $\X$ and points in the hypercompletion $\X^{\hyp}$.
An $\infty$-topos $\X$ is said to have \textit{enough points} if it is possible to detect equivalences by pulling back along points. That is, if $f \colon U \rightarrow V \in \X$ is an equivalence if and only if for each point $x \colon \An \rightarrow \X$ the morphism $x^*f \colon x^* U \rightarrow x^* V$. A topos with enough points is automatically hypercomplete by \cite[Rmk. 6.5.4.7.]{lurie2009higher}.



\begin{remark}\label{rmk:Xhypenoughpoints}
    Let $X$ be any topological space. It is well known that $\Shsp{X}$ is not necessarily hypercomplete  (e.g. take $X$ to be the Hilbert cube), and therefore does not have enough points. However, $\Shsp{X}\hyp$ has enough points for any topological space $X$. This follows from the fact that in $\Shsp{X}$, the class of $\infty$-connective morphisms agrees with that of stalkwise equivalences (see \cite[Lemma A.3.9]{lurie2017higher}).
\end{remark}



\subsection{Internal higher categories}

\label{subsec:internal_higher_cats}

Given a topos $\Y$, Martini-Wolf introduce the notion of a \textit{category internal to $\Y$} as a sheaf of categories on $\Y$, i.e. a limit-preserving functor $\Y\op \rightarrow \Cat$. We get a category
\[ \Cat_{\Y} \coloneqq \Fun^{R}(\Y\op, \Cat) \simeq \Y \otimes \Cat \in \mathrm{Pr}^L \]
whose objects we will shortly refer to as \textit{$\Y$-categories} and whose morphisms will be called \textit{$\Y$-functors}. 

Given a $\Y$-category $\C$ and $U \in \Y$, the value $\C(U)$ will be referred to as the \textit{category of sections of $\C$ over $U$}. In particular, $\C(1)$ will be called the category of \textit{global sections}. The global section functor admits a left adjoint $\mathrm{const} \colon \Cat \rightarrow \Cat_{\Y}$, assigning to a category $I$ its associated \textit{constant $\Y$}-category. 

The category of $\Y$-categories is cartesian closed, we write $\intFun(\C,\D)$ for the \textit{$\Y$-category of $\Y$-functors} from $\C$ to $\D$. We also write $\Fun_{\Y}(\C,\D)$ for its category of global sections, the category of $\Y$-functors from $\C$ to $\D$. Note that the bifunctor $\Fun_{\Y}(\C,\D)$ enhances $\Cat_{\Y}$ to an $(\infty,2)$-category.

Since $\Cat_{\Y}$ is an $(\infty,2)$-category, the notion of \textit{adjunctions} makes sense for $\Y$-categories, which naturally leads to the notions of $\Y$-limits and colimits using the diagonal functor $\C \rightarrow \intFun(I,\C)$ for $\Y$-categories $I$ and $\C$.

The $\Y$-category of \textit{$\Y$-groupoids} is defined as the sheaf $\Omega_\Y \colon U \mapsto \Y_{/U}$. In many ways, $\Omega_\Y$ plays a similar role as that of $\An$ in usual higher category theory. Another useful example of a $\Y$-category is that of locally constant objects. More precisely, this is given by the functor $\Y\op\rightarrow\infcat$ defined by $U\mapsto (\Y_{/U})^{lc}$. We denote this $\Y$-category by $\Omega_{\Y}^{lc}$. There is an obvious fully faithful $\Y$-functor $\Omega_{\Y}^{lc}\hookrightarrow\Omega_{\Y}$.

Let $\C$ be a $\Y$-category. There is a mapping $\Y$-functor $\Map_{\Omega_{\Y}}(-,-) \colon\C\op \times \C \rightarrow \Omega_{\Y}$. For $U \in \Y$ it evaluates to a functor $\Map_{\C}(-,-)_U \colon \C(U)\op \times \C(U) \rightarrow \Y_{/U}$.
In the special case $\C = \Omega_\Y$ there is an identification
\[ \Map_{\Omega_{\Y}}(-,-)_U \simeq \sHom{\Y_{/U}}{-}{-} \colon \Y_{/U}\op \times \Y_{/U} \rightarrow \Y_{/U} \]
with the internal hom induced by the cartesian monoidal structure on $\Y_{/U}$.
Furthermore, for $U \in \Y$ and arbitrary $\C$ there is an equivalence
\[ \Map_{\C}(\pi_U^*(-),\pi_U^*(-))_U \simeq \pi_U^* \Map_{\C}(-,-)_1. \]

\subsection{Internal topos theory}

In \cite{martini2025presentabilitytopoiinternalhigher}, Martini-Wolf introduce the notion of a \textit{topos internal to $\Y$}, which is a certain class of categories internal to $\Y$ satisfying analogs of ordinary definitions of topoi to internal higher category theory. A \textit{geometric morphism of $\Y$-topoi} $F \colon \Theta \rightarrow \Lambda$ is the datum of a right adjoint $\Y$-functor $F_* \colon \Theta \rightarrow \Lambda$ whose left adjoint $F^*$ additionally commutes with finite $\Y$-limits, see \cite[Section 3.2]{martini2025presentabilitytopoiinternalhigher}.
We write $\Top(\Y)$ for the category of $\Y$-topoi.
An example of a $\Y$-topos is the $\Y$-category of $\Y$-groupoids $\Omega_{\Y}$, and in fact this is the final $\Y$-topos. It turns out that, thanks to \cite[Theorem 3.2.5.1]{martini2025presentabilitytopoiinternalhigher}, $\Y$-topoi may be described using the overcategory of $\Y$ in the usual category $\Top$ of topoi and geometric morphisms.
\begin{theorem}
    \label{thm:characterisation_of_internal_topoi}
    Let $\Y$ be a topos. Then taking global sections defines an equivalence of categories
    \[ \Top(\Y) \rightarrow \Top_{/Y}. \]
\end{theorem}

The inverse of the equivalence in \cref{thm:characterisation_of_internal_topoi} is described explicitly in \cite[Remark 3.2.5.3]{martini2025presentabilitytopoiinternalhigher}. Given a geometric morphism of topoi $f \colon \X \rightarrow \Y$, we get an induced functor
\[ f_* \colon \Cat_{\X} \rightarrow \Cat_{\Y}, \hspace{3mm} \C \mapsto [ U \in \Y \mapsto \C(f^* U) ]. \]
Applied to the $\X$-category $\Omega_{\X} \colon V \mapsto \X_{/V}$ this results in the $\Y$-category $f_*\Omega_{\X} \colon U \mapsto \X_{/f^*U}$. This happens to be a $\Y$-topos, it comes with a geometric morphism of $\Y$-topoi $f_* \Omega_X \rightarrow \Omega_Y$, assembling the geometric morphisms $\X_{/f^*U} \rightarrow \Y_{/U}$, which on global sections induces the original geometric morphism $f$.

\begin{notation}
    \label{not:geometric_morphism_notation}
    Given geometric morphisms $\X \xrightarrow{f} \Y \xrightarrow{g} \Z$ of topoi, we will usually write
    $F \colon \pf{(gf)} \Omega_\X \rightarrow \pf{g}\Omega_{\Y}$ for the corresponding geometric morphism of $\Z$-topoi.
\end{notation}

Given a geometric morphism $\X \xrightarrow{f} \Y$, let us also describe the geometric morphism of $\Y$-topoi $F:f_* \Omega_{\X} \rightarrow \Omega_{\Y}$ explicitly at each level $U \in \Y$. The left exact left adjoint $F^* \colon \Omega_{\Y} \rightarrow f^*\Omega_{\X}$ is given at level $U$, by the functor $\Y_{/U} \rightarrow \X_{/f^* U}$ induced by $f^* \colon \Y \rightarrow \X$. By inspection, one sees that the right adjoint $F_* \colon f_*\Omega_{\X} \rightarrow \Omega_{\Y}$ is therefore given at level $U \in \Y$ by the functor $\X_{/f^*U} \rightarrow \Y_{/U}$ sending $V\rightarrow f^*U$ to $f_* V \times_{f_* f^* U} U$.

\subsection{Internally compact objects}

Martini-Wolf introduce a class of \textit{filtered $\Y$-categories} as such $\Y$-categories $J$ such that the colimit functor $\intFun(J,\Omega_{\Y}) \rightarrow \Omega_{\Y}$ preserves $I$-shaped limits, where $I$ is a locally constant sheaf  of finite categories on $\Y$. Colimits over filtered $\Y$-categories will be called \textit{internally filtered colimits}. An object $x$ in a $\Y$-category $\C$ is called \textit{internally compact} if the functor $\Map_{\C}(x,-) \colon \C \rightarrow \Omega_{\Y}$ preserves internally filtered colimits. We write $\C^{cpt}$ for the full subcategory of $\C$ spanned by internally compact objects. There is a useful characterization of internally compact objects in $\Omega_{\Y}$.

\begin{proposition}\label{prop:charinternallycpt=locconstcptvalue}
    Let $\Y$ be an $\infty$-topos and $U \in \Y$ an object. Then $U$ is internally compact if and only if it is locally constant with compact values.
\end{proposition}

\begin{proof}
    This is a reformulation of \cite[Corollary 2.2.3.8]{martini2025presentabilitytopoiinternalhigher}.
\end{proof}

Explicitly, this means that $U$ is internally compact if and only if we can find an \'etale cover $\{V_i \rightarrow 1\in Y\}_{i \in I}$ such that for each $i \in I$ the object $U \times V_i \in \Y_{/V_i}$ is pulled back along the geometric morphism $\Y_{/V_i} \rightarrow \An$ from a compact object $W_i \in \An^\omega$. 

A $\Y$-category $\C$ is said to be \textit{compactly generated} if it is cocomplete and the canonical functor $\underline{\text{Ind}}_{\Y}(\C^{cpt})\rightarrow\C$ is an equivalence. See \cite[2.3.1]{martini2025presentabilitytopoiinternalhigher} and \cite[7.1]{martini2021colimits} for a definition of $\underline{\text{Ind}}_{\Y}$ and colimit completions of $\Y$-categories. The main example of a compactly generated $\Y$-category we care about in this paper is $\Omega_{\Y}$.

\begin{proposition}\label{prop:omegacptlygen}
    The $\Y$-category $\Omega_{\Y}$ is compactly generated.
\end{proposition}

\begin{proof}
    This is a special case of \cite[Proposition 2.3.3.6]{martini2025presentabilitytopoiinternalhigher}
\end{proof}

In practice, \cref{prop:omegacptlygen} is used in conjunction with \cref{prop:charinternallycpt=locconstcptvalue} to verify that a given transformation of internally filtered colimits preserving $\Y$-functors is invertible, by checking it only on locally constant objects.

\subsection{Shape theory}

\begin{definition}
    Let $f:\X\rightarrow\Y$ be any geometric morphism. The \textit{$\Y$-shape} of $f$ is the $\Y$-functor $F_* F^* \colon \Omega_Y \rightarrow \Omega_Y$.
\end{definition}

Thanks to the discussion in the previous subsection, we can give an explicit description of the $\Y$-shape of a geometric morphism $f:\X\rightarrow\Y$. After taking sections at $U\in\Y$, $F_* F^*$ is given by the functor $\Y_{/U}\rightarrow\Y_{/U}$ sending $V\rightarrow U$ to the object $f_*f^* V \times_{f_*f^*U} U \rightarrow U$. The above pullback is taken over the unit transformation $U\rightarrow f_*f^*U$.

\begin{example}\label{example:corepresentableasintshape}
    Let $U$ be an object of $\Y$. The $\Y$-functor $\Map(U,-):\Omega_{\Y}\rightarrow\Omega_{\Y}$ is the $\Y$-shape of the \'etale geometric morphism $\pi_U:\Y_{/U}\rightarrow\Y$. 
\end{example}

When $\Y=\An$, we often make use of the equivalence $\Fun^{\lex,\mathrm{acc}}(\An,\An)\op \simeq \Pro(\An)$ to identify the shape of the geometric morphism $f:\X\rightarrow\An$ with a pro-object in $\An$. This pro-object is denoted by $\sh{\X}$. If $\X=\Sh{X}{\An}$ for some topological space $X$, we write $\sh{X}$ instead of $\sh{\Sh{X}{\An}}$. When $X$ is a compact Hausdorff topological space, $\sh{X}$ agrees with what is classically known as the \textit{strong shape of $X$} (see \cite{hoyois2018higher}). 
The Yoneda embedding induces a fully faithful, finite limit preserving functor
\[ \An \simeq \Fun^{R, \mathrm{acc}}(\An,\An)\op \subset \Pro(\An). \]
Moreover, this functor preserves finite limits. Indeed, from $\Pro(\An) \simeq \mathrm{Ind}(\An\op)\op$ it follows that $\An\op \rightarrow \Pro(\An)\op$ preserves finite colimits \cite[5.3.5.14.]{lurie2009higher}.
We say that an $\infty$-topos $\X$ is of \textit{constant shape} if $\sh{\X}$ is a constant pro-space. In this case $f_* f^* \simeq \Map(\sh{\X},-)$.

For a topological space $X$, one may associate with it the anima $\Pi_{\infty}(X)$, known as the \textit{weak homotopy type} of $X$, or the \textit{singular complex} of $X$. There is always a comparison map 
\begin{equation}\label{eq:compsingtosh}
    \Pi_{\infty}(X)\to\sh{X}
\end{equation} in $\Pro(\An)$, which in general may not be an equivalence (see \cite[Section 3.3]{LandVolpe2026}). Following Lurie, we say that $X$ is \textit{of singular shape} whenever (\ref{eq:compsingtosh}) is an equivalence. Clearly, if $X$ is of singular shape, the $\infty$-topos $\Sh{X}{\An}$ is of constant shape. 

\subsection{Local contractibility}

\label{sec:local_contractibility}

\begin{definition}
    A geometric morphism $f \colon \X \rightarrow \Y$ of topoi is \textit{locally contractible} if the $\Y$-functor $F^* \colon \Omega_{\Y} \rightarrow f_* \Omega_{\X}$ admits a  left adjoint. We say that an $\infty$-topos $\X$ is \textit{locally contractible} if the morphism $\X\rightarrow\An$ is locally contractible. We write $\Top^{loc}$ for the full subcategory of $\Top$ spanned by the locally contractible $\infty$-topoi.
\end{definition}

If $f \colon  \X \rightarrow \Y$ is locally contractible, we generally write $F_\sharp \colon f_* \Omega_{\X} \rightarrow \Omega_{\Y}$ for the left adjoint of $F^*$. Note that on global sections we obtain a left adjoint $f_\sharp \colon \X \rightarrow \Y$ to $f^*$. If $\Y=\An$, we see that $\pfs{f}(1)$ can be identified with $\sh{\X}$.

\begin{remark}\label{rmk:projectionformulaloccontr}
    The condition for $f$ to be locally contractible has a purely external characterization. It is equivalent to $f^*$ admitting a left adjoint $f_\sharp$ that satisfies the projection formula
    \[ f_\sharp(U \times_{f^*W} f^*V) \xrightarrow{\sim} f_\sharp(U) \times_W V \]
    for all $V \rightarrow W$ in $\X$ and $U \rightarrow f^*V$ in $\Y$, see \cite[Remark 3.3.1.3]{martini2025presentabilitytopoiinternalhigher}.
\end{remark}

\begin{example}
    Every \'etale geometric morphism $\pi_U \colon \Y_{/U} \rightarrow \Y$ is locally contractible. In fact, \'etale geometric morphisms to $\Y$ are characterized as precisely those locally contractible geometric morphisms $f \colon \X \rightarrow \Y$ for which $f_\sharp \colon \X \rightarrow \Y$ is conservative, see \cite[Example 3.3.1.4]{martini2025presentabilitytopoiinternalhigher}.
\end{example}

An important example of locally contractible $\infty$-topoi we care about for the purpose of this paper is $\Sh{X}{\An}$ when $X$ is an \textit{absolute neighborhood retract} (ANR) (see the discussion after \cite[Theorem 7.1]{Dugundji1951} for a definition). 

\begin{lemma}
    \label{ex:ANRs}
    Let $X$ be an ANR. Then $\Sh{X}{\An}$ is locally contractible. Moreover, $X$ is of singular shape. 
\end{lemma}

\begin{proof}
    Since any open subset of an ANR is again ANR, to show that $\Sh{X}{\An}$ is locally contractible it suffices to prove that $X$ is of singular shape. By homotopy invariance of $\Pi_{\infty}$ and $\sh{-}$, and naturality of (\ref{eq:compsingtosh}), we see that it suffices to find a space $Y$ of singular shape, continuous maps $i\colon X\to Y$, $r\colon Y\to X$ and a homotopy $ri\simeq\mathrm{id}_X$. This datum is provided by \cite[Theorem 7.5]{Dugundji1951}.
\end{proof}


Suppose that $f:\X \rightarrow \An$ is locally contractible. Then the equivalence $\X\simeq\X_{/1}$ implies that $f_\sharp$ factors through a functor $\pfs{\eta}:\X\rightarrow\An_{/\sh{\X}}$
As discussed in \cite[p. 1405]{lurie2017higher}, $\eta_\sharp$ admits a right adjoint $\eta^*$, which can be described informally with the formula $\eta^*(y) \simeq f^*(y) \times_{f^*f_\sharp(1)} 1 $. From this description and universality of colimits, it is clear that $\eta^*$ is cocontinuous. Hence $\pb{\eta}$  has a further right adjoint $\eta_*$, and the pair $ (\eta^*,\eta_*)$ describes a geometric morphism
\[ \eta \colon \X \rightarrow \An_{/\sh{\X}}. \]
The projection formula for $\pfs{\eta}$ can be deduced from the one of $\pfs{f}$, and therefore $\eta$ is a locally contractible geometric morphism. The main result of \cite[Appendix A.1]{lurie2017higher} shows that $\pb{\eta}$ is fully faithful, with essential image given by the locally constant objects in $\X$.

\begin{remark}\label{rmk:etanatural}
    The morphism $\eta$ is in fact part of a natural transformation $$\eta:\text{id}\rightarrow\Fun(\sh{(-)},\An)$$ of functors $\Top^{loc}\rightarrow\Top^{loc}$. This follows for example from the discussion in \cite[Page 4]{hoyois2018higher}.
    If $f \colon \X \rightarrow \Y$ is a morphism of locally contractible topoi, we write $\sh{f} \colon \An_{/\sh{\X}} \rightarrow \An_{/\sh{\Y}}$ for the induced geometric morphism. If we want to consider $\sh{f}$ as a geometric morphism of $\An_{/\sh{Y}}$-topoi, we write it as $\sh{F} \colon \pf{\sh{f}} \Omega_{\An_{/\sh{X}}} \rightarrow \Omega_{\An_{/\sh{\Y}}}$, in accordance with \cref{not:geometric_morphism_notation}.
\end{remark}

\begin{lemma}\label{lemma:shomlocconstislocconst}
    Let $\X$ be a locally of constant shape $\infty$-topos, and let $F$ and $G$ be any two functors $\sh{\X}\rightarrow\An$. Then the canonical map
    $$\pb{\eta}\sHom{\An_{/\sh{\X}}}{F}{G}\rightarrow\sHom{\X}{\pb{\eta}F}{\pb{\eta}G}$$
    is an isomorphism. In particular, we get that $\sHom{\X}{\pb{\eta}F}{\pb{\eta}G}$ is locally constant.
\end{lemma}

\begin{proof}
    This follows immediately from the projection formula for $\pfs{\eta}$.
\end{proof}

\begin{lemma}\label{Xloccontr=>locconsthyp}
    Let $\X$ be a locally contractible $\infty$-topos. Then locally constant objects in $\X$ are hypercomplete. 
\end{lemma}

\begin{proof}
    This follows from \cite[Corollary A.1.17]{lurie2017higher}.
\end{proof}

\begin{remark}\label{rmk:cptvalues&cptstalks}
    Let $\X$ be a locally contractible $\infty$-topos, and let $\sh{\X}\xrightarrow{F}\An$ be any functor. Then $\pb{\eta}(F)\in\X$ has compact values if and only if, for any point $[0]\xrightarrow{x}\sh{\X}$, $\pb{x}F\in\An$ is a compact object.
\end{remark}

\section{Fiber theorems}\label{sec:fiber_theorems}

In this section, we will study a relative generalization of that of a cell-like map in the situation of a commuting diagram of $\infty$-topoi
\begin{center}
    \begin{tikzcd}
        \X \ar[rr, "f"] \ar[dr, "p"'] && \X' \ar[dl, "q"]\\
        & \Y &
    \end{tikzcd}
\end{center}
such that when $q$ is an equivalence, we recover Lurie's notion of cell-like maps. Apart from merely generalizing, our methods also give a new proof of Lurie's characterization of maps between ANRs inducing cell-like maps on sheaf topoi, which is actually more general as well. To motivate, let us recall the notion of a shape equivalence of $\infty$-topoi.

Let $f \colon \X \rightarrow \X'$ be a geometric morphism. If $p \colon \X \rightarrow \An$ and $q \colon \X' \rightarrow \An$ denote the unique geometric morphisms, then since $qf \simeq p$ we get a map 
\begin{equation*}
    \pf{q}\pb{q}\xrightarrow{\text{unit}}\pf{q}\pf{f}\pb{f}\pb{q}\simeq \pf{p}\pb{p}
\end{equation*}
which corresponds under $\Fun^{\mathrm{acc,lex}}(\An,\An)\op \simeq \Pro(\An)$ to a map $\sh{\X} \rightarrow \sh{\X'} \in \Pro(\An)$. Following \cite[Rem. 7.1.6.5.]{lurie2009higher} we say that $f$ is a \textit{shape equivalence} if either the natural transformation written above or equivalently the map $\sh{\X} \rightarrow \sh{\X'}$ is an equivalence. It turns out to be more straightforward to generalize the first perspective to a general base.

In the following, we will often work in the category $\Top_{/Y}$ for some $\infty$-topos $\Y$. The main result of \cite{Martini2023a} identifies this category with the category of $\infty$-topoi \textit{internal} to $\Y$, a perspective that is often valuable. For a point $y \colon \An \rightarrow \Y$ and $\X \in \Top_{/\Y}$ we will use the notation $\X_y$ for the pullback $\X \times_{\Y} \An$ along the point $y$ and call it \textit{the fiber of $\X$ over $y$}. One of our goals is to study the aforementioned relative generalization of the notion of a shape equivalence and how it interacts with passing to fibers over $\Y$.

\begin{definition}\label{def:relshapeeq}
    Let $\Y$ be an $\infty$-topos, and let $\X\xrightarrow{p}\Y$, $\X'\xrightarrow{q}\Y$ be geometric morphisms. Let $\X\xrightarrow{f}\X'$ be a geometric morphism over $\Y$. We say that $f$ is a \textit{relative shape equivalence over $\Y$} if the natural transformation
        \begin{equation}\label{relativeshapeeq}
            \pf{q}\pb{q}\xrightarrow{\text{unit}}\pf{q}\pf{f}\pb{f}\pb{q}\simeq \pf{p}\pb{p}
        \end{equation}
        is invertible.
\end{definition}

For $\Y = \An$, this notion specialises to the classical notion of a shape equivalence.

\begin{remark}\label{rmk:cell-likeisrelshapeeq}
    Suppose we are in the same situation as in \cref{def:relshapeeq}. When $\X'=\Y$ and $q=\text{id}_{\Y}$, we see that $f$ is a relative shape equivalence over $\Y$ if and only if $\pb{f}$ is fully faithful. Therefore, when $f$ is also assumed to be proper, we see that $f$ is a relative shape equivalence over $\Y$ if and only if $f$ is \textit{cell-like} in the sense of Lurie \cite[Def. 7.3.6.1.]{lurie2009higher}.
\end{remark}

\begin{definition}\label{def:hershapeq/Y}
    Let $\Y$ be an $\infty$-topos, and let $\X\xrightarrow{p}\Y$, $\X'\xrightarrow{q}\Y$ be geometric morphisms. Let $\X\xrightarrow{f}\X'$ be a geometric morphism over $\Y$. We say that $f$ is a \textit{hereditary shape equivalence over $\Y$} if, for any $U\in\Y$, the geometric morphism
    $$\X_{/\pb{p}U}\xrightarrow{f_{/U}}\X'_{/\pb{q}U}$$
    is a shape equivalence. When $\X'=\Y$ and $q=\text{id}_{\Y}$, we simply say $f$ is a \textit{hereditary shape equivalence}.
\end{definition}

We emphasize that we really mean the map $\X_{/\pb{p}U}\xrightarrow{f_{/U}}\X'_{/\pb{q}U}$ to be a shape equivalence relative to $\An$ and \textit{not} relative to $\Y$, although it certainly lives in $\Top_{/Y}$.

\begin{lemma}\label{lem:hereditaryshapeeqongenerators}
	Let $\Y$ be an $\infty$-topos, and let $\X\xrightarrow{p}\Y$, $\X'\xrightarrow{q}\Y$ be geometric morphisms. Let $\X\xrightarrow{f}\X'$ be a geometric morphism over $\Y$. Let $\mathcal{B}\subseteq\Y$ be a family of objects generating $\Y$ under colimits. Then $f$ is a hereditary shape equivalence over $\Y$ if and only if, for each $V\in\mathcal{B}$, the geometric morphism $\X_{/\pb{p}V}\xrightarrow{f_{/V}}\X'_{/\pb{q}V}$ is a shape equivalence. 
\end{lemma}

\begin{proof}
    For any $U\in\Y$, using descent we may write the map $\X_{/\pb{p}U}\xrightarrow{f_{/U}}\X'_{/\pb{q}U}$ as a colimit of maps of the form $f_{/V}$, where $V$ lies in $\mathcal{B}$. The claim then follows from the fact that the functor $\text{sh}: \Top\rightarrow\Pro(\An)$ preserves colimits.
\end{proof}
    
\begin{corollary}
    Let $Y$ be any topological space, and let $\Y=\Shsp{Y}$ be the $\infty$-topos of sheaves of animae on $Y$. Let $\X\xrightarrow{p}\Y$, $\X'\xrightarrow{q}\Y$ be geometric morphisms, and let $\X\xrightarrow{f}\X'$ be a geometric morphism over $\Y$. Then $f$ is a hereditary shape equivalence over $\Y$ if and only if, for any $U\hookrightarrow Y$ open subset, the geometric morphism $f_{/U}: \X_{/\pb{p}U}\rightarrow\X'_{/\pb{q}U}$ is a shape equivalence.
\end{corollary}

\begin{proof}
    This follows from \cref{lem:hereditaryshapeeqongenerators} and the fact that $\Y$ is generated by the sheaves represented by the open subsets of $Y$. 
\end{proof}

\begin{lemma}\label{hershapeeqiffrelunitinvonconstants} Let $\Y$ be an $\infty$-topos, and let $\X\xrightarrow{p}\Y$, $\X'\xrightarrow{q}\Y$ be geometric morphisms. Let $\X\xrightarrow{f}\X'$ be a geometric morphism over $\Y$. Then the following are equivalent.

\begin{enumerate}[(1)]
    \item $f$ is a hereditary shape equivalence over $\Y$;
    \item For any constant object $F\in\Y$, the map 
    $$\pf{q}\pb{q}F\rightarrow\pf{p}\pb{p}F$$
    obtained by evaluating the natural transformation (\ref{relativeshapeeq}) at $F$ is invertible.
\end{enumerate}
\end{lemma}

\begin{proof}
    Indeed, let $U$ be any object in $\Y$, and consider the diagram 
    $$
    \begin{tikzcd}
\X_{/\pb{p}U} \arrow[d, "\rho_U"] \arrow[r, "p_{/U}"] & \Y_{/U} \arrow[d, "\pi_U"] & \X'_{/\pb{q}U} \arrow[d, "\sigma_U"] \arrow[l, "q_{/U}"'] \\
\X \arrow[rd, "b"'] \arrow[r, "p"]                    & \Y \arrow[d, "a"]          & \X' \arrow[ld, "c"] \arrow[l, "q"']                       \\
                                                      & \An                       &                                                          
\end{tikzcd}
    $$ 
    where $\pi_U$ is the \'etale geometric morphism corresponding to $U$, the squares are pullbacks of $\infty$-topoi, and the lower arrows pointing towards $\An$ are the unique geometric morphisms. Suppose $A$ is any anima, and let $\pb{a}A\in\Y$ be the associated constant object in $\Y$. The map $\pf{q}\pb{q}\pb{a}A\xrightarrow{\text{unit}} \pf{p}\pb{p}\pb{a}A$ is invertible if and only if it is invertible after taking sections at $U$, for all $U\in\Y$. Note that taking sections at $U \in Y$ means to restrict to $\Y_{/U}$ and apply the unique geometric morphism $\Y_{/U} \rightarrow \An$, which is the composition $\pf{a} \pf{{\pi_U}} \pb{{\pi_U}}$. We have a commuting diagram
    $$
    \begin{tikzcd}
\pf{a}\pf{{\pi_U}}\pb{{\pi_U}}\pf{q}\pb{q}\pb{a}A \arrow[d, "\simeq"] \arrow[rr] &  & \pf{a}\pf{{\pi_U}}\pb{{\pi_U}}\pf{p}\pb{p}\pb{a}A \arrow[d, "\simeq"]      \\
\pf{a}\pf{{\pi_U}}\pf{{q_{/U}}}\pb{\sigma_U}\pb{q}\pb{a}A \arrow[d, "\simeq"]     &  & \pf{a}\pf{{\pi_U}}\pf{{p_{/U}}}\pb{\rho_U}\pb{p}\pb{a}A \arrow[d, "\simeq"] \\
\pf{c}\pf{{\sigma_U}}\pb{{\sigma_U}}\pb{c}A \arrow[rr]                                                         &  & \pf{b}\pf{{\rho_U}}\pb{{\rho_U}}\pb{b}A                                                            
\end{tikzcd}
    $$
    where the upper horizontal arrow is obtained by applying $\pf{a}\pf{{\pi_U}}\pb{{\pi_U}}$ to (\ref{relativeshapeeq}), the lower horizontal arrow is defined analogously to (\ref{relativeshapeeq}) via the unit of the adjunction $\pb{{f_{/U}}}\dashv\pf{{f_{/U}}}$, and the first vertical isomorphisms are given by \'etale basechange (see \cite[Lemma 3.2.1.9]{martini2025presentabilitytopoiinternalhigher}). Therefore, the upper horizontal map is invertible for all $U\in\Y$ and $A\in\An$ if and only if $f_{/U}$ is a shape equivalence for all $U\in\Y$.
\end{proof}

\begin{definition}\label{def:ptwisesheq}
    Let $\Y$ be an $\infty$-topos, and let $\X\xrightarrow{p}\Y$, $\X'\xrightarrow{q}\Y$ be geometric morphisms. Let $\X\xrightarrow{f}\X'$ be a geometric morphism over $\Y$. We say that $f$ is a \textit{pointwise shape equivalence over $\Y$} if, for any point $\An\xrightarrow{y}\Y$, the geometric morphism 
    $$\X_{y}\xrightarrow{f_y}\X'_y$$
    is a shape equivalence.
\end{definition}

\begin{remark}\label{rmk:relsheqimpliesptwisesheq}
    Let $\Y$ be an $\infty$-topos, and let $\X\xrightarrow{p}\Y$, $\X'\xrightarrow{q}\Y$ be proper geometric morphisms. Let $\X\xrightarrow{f}\X'$ be a relative shape equivalence over $\Y$. Then $f$ is a pointwise shape equivalence over $\Y$.  The proof of our claim is identical to that of the only if part of \cite[Proposition 7.3.6.4]{lurie2009higher}, so we omit it. This is essentially a consequence of proper basechange and some additional small observations.
\end{remark}

\cref{rmk:relsheqimpliesptwisesheq} has the following partial converse.

\begin{lemma}\label{lem:partconvrmk:relsheqimpliesptwisesheq}
    Let $\Y$ be an $\infty$-topos, and let $\X\xrightarrow{p}\Y$, $\X'\xrightarrow{q}\Y$ be proper geometric morphisms. Let $\X\xrightarrow{f}\X'$ be a pointwise shape equivalence over $\Y$. Assume that $\Y\hyp$ has enough points, and that $F\in\Y$ is an object such that $\pb{p}F$ and $\pb{q}F$ are hypercomplete. Then the morphism
    $$\pf{q}\pb{q}F\rightarrow\pf{p}\pb{p}F$$
    is invertible.
\end{lemma}

\begin{proof}
    Since $\pb{p}F$ and $\pb{q}F$ are hypercomplete, and lower stars preserve hypercompleteness, we have that $\pf{q}\pb{q}F$ and $\pf{p}\pb{p}F$ are hypercomplete. Therefore, since $\Y\hyp$ has enough points, to prove the lemma we only need to show that the map $\pf{q}\pb{q}F\rightarrow\pf{p}\pb{p}F$ is invertible after pulling back along any point of $\Y$. Using proper basechange, one sees that this last statement follows immediately from the assumption that $f$ is a pointwise shape equivalence over $\Y$.
\end{proof}

\begin{theorem}\label{thm:proper.relsheqiffhersheq}
    Let $\Y$ be an $\infty$-topos, and let $\X\xrightarrow{p}\Y$, $\X'\xrightarrow{q}\Y$ be proper geometric morphisms. Let $\X\xrightarrow{f}\X'$ be a geometric morphism over $\Y$. Consider the following conditions.
    \begin{enumerate}[(1)]
        \item $f$ is a relative shape equivalence over $\Y$;
        \item $f$ is a hereditary shape equivalence over $\Y$;
        \item $f$ is a pointwise shape equivalence over $\Y$.
    \end{enumerate}
    Then (1) is equivalent to (2), and (1) implies (3). Moreover, if constant objects in $\X$ and $\X'$ are hypercomplete and $\Y\hyp$ has enough points, then (1) and (2) are equivalent to (3). 
\end{theorem}

\begin{proof}
    It follows respectively from \cref{hershapeeqiffrelunitinvonconstants} and \cref{rmk:relsheqimpliesptwisesheq} that (1) implies (2) and (3). 
    
    We now show that (2) implies (1). The morphism (\ref{relativeshapeeq}) induces a transformation of $\Y$-functors 
    \begin{equation}\label{relativeshapeeqinternal}
    	\pf{Q}\pb{Q}\rightarrow\pf{P}\pb{P}.
    \end{equation}
    Since both $p$ and $q$ are proper, by \cite{Martini2023a} the source and target of the transformation above are $\Y$-functors preserving internally filtered colimits. Therefore, it suffices to show that (\ref{relativeshapeeqinternal}) is invertible after restricting to the $\Y$-category of internally compact objects in $\Omega_{\Y}$. By \cref{prop:charinternallycpt=locconstcptvalue}, this condition amounts to proving that, for any $U\in\Y$ and $F\in\Y_{/U}$ locally constant with compact values, the map $$\pf{{q_{/U}}}\pb{{q_{/U}}}F\rightarrow\pf{{p_{/U}}}\pb{{p_{/U}}}F$$
    is invertible. Using \'etale basechange we may further reduce to the case when $F$ is constant. Appealing to \cref{hershapeeqiffrelunitinvonconstants}, one sees that in this last case the map $\pf{q}\pb{q}F\rightarrow\pf{p}\pb{p}F$ is guaranteed to be invertible by condition (2). Therefore we deduce that (1) holds.

    Now assume that constant objects in $\Y$, $\X$ and $\X'$ are hypercomplete. To conclude the proof of the theorem, it suffices to show that (3) implies (2). More specifically, by \cref{hershapeeqiffrelunitinvonconstants}, we need to prove that, for any $F\in\Y$ constant, the map 
    $$\pf{q}\pb{q}F\rightarrow\pf{p}\pb{p}F$$
    is invertible. Since $F$ is constant, both $\pb{p}F$ and $\pb{q}F$ are constant, and therefore hypercomplete by assumption. Hence our proof is concluded by appealing to \cref{lem:partconvrmk:relsheqimpliesptwisesheq}.
\end{proof}

\begin{corollary}
    \label{cor:cell_like_characterization}
    Let $\X\xrightarrow{f}\Y$ be a proper geometric morphism. Consider the following conditions.
    \begin{enumerate}[(1)]
        \item $f$ is cell-like;
        \item $f$ is a hereditary shape equivalence;
        \item for each point $\An\xrightarrow{y}\Y$, $\X_y$ has trivial shape.
    \end{enumerate}
    Then (1) is equivalent to (2), and (1) implies (3). Moreover, if constant objects in $\X$ and $\Y$ are hypercomplete and $\Y\hyp$ has enough points, then (1) and (2) are equivalent to (3). 
\end{corollary}

\begin{proof}
    This is an immediate consequence of \cref{thm:proper.relsheqiffhersheq} and \cref{rmk:cell-likeisrelshapeeq}.
\end{proof}

\begin{remark}\label{rmk:generalizelurie}
    We observe that our result is a generalization of Lurie's characterization of cell-like maps between ANRs in \cite[Proposition 7.3.6.6]{lurie2009higher}.\footnote{In the cited reference, Lurie only proves the characterization for locally compact ANRs. The assumption of local compactness can be dropped thanks to \cite[Theorem 3.5.1, Corollary 3.5.4]{Martini2023a}. We point out that metrizable spaces are sober and completely regular, and any proper map with completely regular source is separated, so that in our situation one may ignore the additional assumptions in \cite[Corollary 3.5.4]{Martini2023a}.}
    Indeed, by \cref{ex:ANRs}, the notion of a hereditary shape equivalence for ANRs precisely translates to condition (2) therein, and the condition on fibers having trivial shape translates to condition (3) by the proof of \cref{cor:appendixshapefibers}. So we have to justify that if $f \colon \X \rightarrow \Y$ is induced by a proper map of ANRs, all three conditions in \cref{cor:cell_like_characterization} are equivalent. This follows since $\Y^{\hyp}$ has enough points by \cref{rmk:Xhypenoughpoints}, and constant objects in $\X$ and $\Y$ are hypercomplete by \cref{Xloccontr=>locconsthyp}. 
\end{remark}

\begin{remark}
    If $\X$ and $\Y$ are locally contractible $\infty$-topoi, the results of \cref{cor:cell_like_characterization} can also be proved by the methods of the next section. See \cref{rmk:cell-likefromapproxfibloccontr} for a more detailed discussion.
\end{remark}

\section{Approximate fibrations}

In this section, we give the definition of an approximate fibration of $\infty$-topoi. The goal is to put it into a shape-theoretic context, so we start by investigating a more intuitive notion, namely that of a shape fibration. This is essential to compare our definition to the classical notion of approximate fibration by Coram and Duvall \cite{CoramDuvall}, which rounds off the section.

\subsection{Shape fibrations and quasi-fibrations}

\label{subsec:shape_fibrations}

The shape of an $\infty$-topos is one of its fundamental invariants, but as a tool for studying maps it is often too coarse. Given a continuous map $f \colon X \to Y$, the induced morphism of shapes $\sh{f} \colon \sh{X} \to \sh{Y}$ typically retains only very limited information about the behaviour of $f$.  


To extract finer information from a map $f$, one should not only consider $\sh{f}$ itself, but also the induced maps $\sh{f^{-1}(U)} \longrightarrow \sh{U}$ for open subsets $U \subset Y$. This leads to the notion of a \emph{shape fibration}: a map $f$ is called a shape fibration if, for every open $U \subset Y$, the map $\sh{f^{-1}(U)} \to \sh{U}$ arises as a base change of $\sh{f}$ along $\sh{U} \to \sh{Y}$.  

In what follows, we introduce the concept of shape fibration for geometric morphisms of arbitrary $\infty$-topoi, study its behaviour under local contractibility assumptions, and compare it with the weaker notion of \emph{shape quasi-fibration}, which depends only on the fibres over points. We show that every proper shape fibration is a shape quasi-fibration, and we identify conditions under which the converse holds.

\begin{definition}\label{def:shapefib}
    Let $\X\xrightarrow{f}\Y$ be a geometric morphism. We say that $f$ is a \textit{shape fibration} if, for any $U\in\Y$, the square
    $$
    \begin{tikzcd}
        \sh{\X_{/\pb{f}U}}\arrow[r]\arrow[d] & \sh{\X}\arrow[d, "{\sh{f}}"] \\
        \sh{\Y_{/U}}\arrow[r] & \sh{\Y}
    \end{tikzcd}
    $$
    is a pullback in $\Pro(\An)$.
\end{definition}

\begin{remark}\label{rmk:shape(quasi)fibstableunderetalebasechange}
    Let $f:\X\rightarrow\Y$ be a shape fibration, let $U\in\Y$ be any object, and let $f_{/U}:\X_{/\pb{f}U}\rightarrow\Y_{/U}$ be the morphism obtained by pulling back $f$ along the \'etale morphism $\Y_{/U}\rightarrow\Y$. Then $f_{/U}$ is a shape fibration. Indeed, since $f$ is a shape fibration, the composite square and the right square in the diagram
    $$
    \begin{tikzcd}
\sh{\X_{/\pb{f}V}} \arrow[r] \arrow[d] & \sh{\X_{/\pb{f}U}} \arrow[d, "\sh{f_{/U}}"] \arrow[r] & \sh{\X} \arrow[d, "\sh{f}"] \\
\sh{\Y_{/V}} \arrow[r]                 & \sh{\Y_{/U}} \arrow[r]                                & \sh{\Y}                   
\end{tikzcd}
    $$
    are cartesian. Therefore the left square is also cartesian. Our claim is proven by observing that, for any map $V\rightarrow U$, there is an equivalence of $\infty$-categories $(\Y_{/U})_{/V}\simeq \Y_{/V}$.  
\end{remark}

\begin{lemma}\label{lem:shapefibsufficescheckongenerators} 
Let $\X$ and $\Y$ be locally contractible $\infty$-topoi, and let $f:\X\rightarrow\Y$ be any geometric morphism. Suppose that $\mathcal{B}\subseteq\Y$ is a family of objects generating $\Y$ under colimits. Then $f$ is a shape fibration if and only if, for any $U\in\mathcal{B}$ the square 
	$$
	\begin{tikzcd}
		\sh{\X_{/\pb{f}U}}\arrow[r]\arrow[d] & \sh{\X}\arrow[d, "{\sh{f}}"] \\
		\sh{\Y_{/U}}\arrow[r] & \sh{\Y}
	\end{tikzcd}
	$$
	is a pullback in $\Pro(\An)$.
\end{lemma}

\begin{proof}
    Let $V\in\Y$ be any object in $\Y$. We need to show that the map $\sh{\X_{/\pb{f}V}}\rightarrow\sh{\Y_{/V}}\times_{\sh{\Y}}\sh{\X}$ in $\Pro(\An)$ is invertible.
    Notice that, since $\X$ and $\Y$ are assumed to be locally contractible, and \'etale morphisms are locally contractible, the map above lies in the full subcategory $\An\hookrightarrow\Pro(\An)$. By assumption, we can write $V\simeq\varinjlim\limits_{i\in I}U_i$, where $I\rightarrow\Y$ is a functor with $U_i\in\mathcal{B}$ for all $i\in I$. Therefore, we obtain an isomorphism $\sh{Y_{/V}}\simeq\varinjlim\limits_{i\in I}\sh{Y_{/U_i}}$ in $\An$. Using respectively universality of colimits in $\An$, our hypothesis for $f$ and the fact that $\pb{f}$ preserves colimits, we deduce
    \begin{align*}
        \sh{\Y_{/V}}\times_{\sh{\Y}}\sh{\X}&\simeq (\varinjlim\limits_{i\in I}\sh{Y_{/U_i}})\times_{\sh{\Y}}\sh{\X} \\
        &\simeq \varinjlim\limits_{i\in I}(\sh{Y_{/U_i}}\times_{\sh{\Y}}\sh{\X}) \\
        &\simeq \varinjlim\limits_{i\in I}\sh{\X_{/\pb{f}U_i}} \\
        &\simeq \sh{\X_{/\pb{f}V}}.\qedhere
    \end{align*}
\end{proof}

\begin{corollary}\label{cor:shapefibontopspace}
    Let $\X$ be any $\infty$-topos, let $Y$ be any topological space, and assume that $\X$ and $\Y=\Shsp{Y}$ are locally contractible. Let $f:\X\rightarrow\Y$ be any geometric morphism. Then $f$ is a shape fibration if and only if, for any open subset $U\hookrightarrow Y$, the square 
    $$
    \begin{tikzcd}
        \sh{\X_{/\pb{f}U}}\arrow[r]\arrow[d] & \sh{\X} \arrow[d, "\sh{f}"] \\
        \sh{U} \arrow[r] & \sh{Y}
    \end{tikzcd}
    $$
    is a pullback in $\An$.
\end{corollary}

\begin{proof}
    This follows immediately from \cref{lem:shapefibsufficescheckongenerators} and the fact that $\Shsp{Y}$ is generated by the sheaves represented by the open subsets of $Y$.
\end{proof}

\begin{theorem}\label{thm:shapefibsandpreslocconst}
    Let $\X\xrightarrow{f}\Y$ be any geometric morphism and assume that $\X$ and $\Y$ are locally contractible $\infty$-topoi. Then the following are equivalent.
    \begin{enumerate}[(1)]
        \item $f$ is a shape fibration;
        \item the functor $\pf{f}\pb{f}$ preserves locally constant objects.
    \end{enumerate}
\end{theorem}

\begin{proof}
    Since $\X$ and $\Y$ are locally contractible, as explained in \cref{sec:local_contractibility} we may consider the commuting diagram
    \begin{equation*}
        \begin{tikzcd}
            \X \ar[r, "\widetilde{\eta}"] \ar[d, "f"] & \An_{/\sh{\X}} \ar[d, "\sh{f}"]\\
            \Y \ar[r, "\eta"] & \An_{/\sh{Y}}
        \end{tikzcd}
    \end{equation*}
    The counit of the adjunction $\pb{\eta}\dashv \pf{\eta}$ produces a natural transformation
    \begin{equation}\label{comppreslocconst}
        \pb{\eta} \pf{\eta}\pf{f}\pb{f}\pb{\eta}\rightarrow\pf{f}\pb{f}\pb{\eta}.
    \end{equation}
    Clearly, (2) holds if and only if (\ref{comppreslocconst}) is invertible. Notice that, by \cref{rmk:etanatural}, the domain of the transformation (\ref{comppreslocconst}) is isomorphic to $\pb{\eta}\pf{\sh{f}}\pb{\sh{f}}$.
    The morphism (\ref{comppreslocconst}) is invertible if and only if it is invertible after taking sections at each $U\in\Y$. Then, the invertibility of (\ref{comppreslocconst}) is seen to be equivalent to the Beck-Chevalley map of the following diagram
    \begin{equation*}
        \begin{tikzcd}
            \An_{/\sh{\X_{/\pb{f}U}}} \ar[d, "\pf{\sh{f_{/U}}}"] & \An_{/\sh{\X}} \ar[d, "\pf{\sh{f}}"] \ar[l, "\pb{\sh{\pi_U}}"'] & \An \ar[l, "\pb{\pi_{\sh{X}}}"'] \\
             \An_{/\sh{\Y_{/U}}}  & \An_{/\sh{\Y}} \ar[l, "\pb{\sh{\pi_U}}"'] &
        \end{tikzcd}
    \end{equation*}  
    being invertible after applying sections $\pf{(\pi_{\sh{\Y_{/U}}})} \colon \An_{/\sh{\Y_{/U}}} \rightarrow \An$ for every $U \in \Y$.
    Denote by $P_U$ the pullback $\sh{\Y_{/U}}\times_{\sh{\Y}}\sh{\X}$, and let $c \colon \sh{\X_{/\pb{f}U}} \rightarrow P_U$ be the comparison map. By \'etale base change, the map
    \[  \pf{(\pi_{\sh{\Y_{/U}}})} \pb{\sh{\pi_U}} \pf{\sh{f}} \pb{\pi_{\sh{X}}} \rightarrow \pf{(\pi_{\sh{\Y_{/U}}})}  \pf{\sh{f_{/U}}} \pb{\sh{\pi_U}} \pb{\pi_{\sh{X}}} \]
    identifies with the map $c^* \colon \Map(P_U,Z) \rightarrow \Map(\sh{\X_{/U}},Z)$. This map is an equivalence for all $Z$ if and only if $c$ is an equivalence. 
\end{proof}

\begin{definition}\label{def:shapequasifib}
    Let $\X\xrightarrow{f}\Y$ be a geometric morphism. We say that $f$ is a \textit{shape quasi-fibration} if, for any point $\An\xrightarrow{y}\Y$, the square
    $$
    \begin{tikzcd}
        \sh{\X_y}\arrow[r]\arrow[d] & \sh{\X}\arrow[d, "{\sh{f}}"] \\
        \sh{\An}\arrow[r, "y"] & \sh{\Y}
    \end{tikzcd}
    $$
    is a pullback in $\Pro(\An)$.
\end{definition}

\begin{lemma}\label{lem:quasi-fibbtweenconstshap=>constshapefib}
    Let $\X\xrightarrow{f}\Y$ be any shape quasi-fibration, and assume that both $\sh{\X}$ and $\sh{\Y}$ are constant pro-objects. Then, for any point $\An\xrightarrow{y}\Y$, $\sh{\X_y}$ is a constant pro-object. Moreover, if one additionally assumes that $f$ is proper, then $\sh{\X_y}$ is a compact object in $\An$.
\end{lemma}

\begin{proof}
    Since the inclusion $\An \subset \Pro(\An)$ induced by the Yoneda embedding preserves finite limits, we see that if $\sh{\X}$ and $\sh{Y}$ are constant pro-objects, so is $\sh{\X_y}$.

    If one also assumes that $f$ is proper, we get that $\X_y$ is a proper $\infty$-topos with constant shape. Then arguing exactly as in \cite[Proposition A.6]{volpe2025verdier}, one sees that $\sh{\X_y}$ must be compact, as it corepresents a functor that preserves filtered colimits.
\end{proof}

\begin{lemma}\label{lemma:approxfib=>shapequasifib}
    Let $\X\xrightarrow{f}\Y$ be a proper geometric morphism, and suppose that $f$ is a shape fibration. Then $f$ is a shape quasi-fibration.
\end{lemma}

\begin{proof}
    Let $\An\xrightarrow{y}\Y$ be any point in $\Y$ and let $\PSh{C}\xrightarrow{L}\Y$ be any left exact localization where $C$ is small. Denote by $N_{C}(y)$ the $\infty$-category of $C$-neighbourhoods of $y$ (see \cref{def:C-neighbourhoods}). Combining \cref{cor:shapeoffibpropermorph} and \cref{rmk:N(x)isfilt}, one sees that the square appearing in \cref{def:shapequasifib} is a limit of squares as follows
    $$
    \begin{tikzcd}
        \varprojlim\limits_{U\in N_{C}(y)\op}\sh{\X_{/\pb{f}U}}\arrow[r]\arrow[d] & \varprojlim\limits_{U\in N_{C}(y)\op}\sh{\X}\arrow[d, "{\sh{f}}"] \\
        \varprojlim\limits_{U\in N_{C}(y)\op}\sh{\Y_{/U}}\arrow[r] & \varprojlim\limits_{U\in N_{C}(y)\op}\sh{\Y}.
    \end{tikzcd}
    $$
    Since $f$ is a shape fibration, for each $U\in N_C(y)$, the corresponding square is a pullback in $\Pro(\An)$. Therefore the proof is concluded by observing that a limit of pullback squares is again a pullback. 
\end{proof}

\begin{theorem}\label{thm:quasifibsandpreslocconst}
    Let $\X\xrightarrow{f}\Y$ be any proper geometric morphism, and assume that $\X$ and $\Y$ are locally contractible $\infty$-topoi. Suppose additionally that $\Y\hyp$ has enough points. Then the following are equivalent.
    \begin{enumerate}[(1)]
        \item $f$ is a shape fibration;
        \item $f$ is a shape quasi-fibration;
        \item the functor $\pf{f}\pb{f}$ preserves locally constant objects.
    \end{enumerate}
\end{theorem}

\begin{proof}
    By \cref{thm:shapefibsandpreslocconst} we already know that (1) is equivalent to (3), and by \cref{lemma:approxfib=>shapequasifib} we also know that (1) implies (2). Therefore, we are only left to show that (2) implies (3). As observed earlier in the proof of \cref{thm:shapefibsandpreslocconst}, (3) holds if and only if the natural transformation (\ref{comppreslocconst}) is invertible. Since $\Y\hyp$ has enough points, and locally constant objects in $\X$ and $\Y$ are hypercomplete by \cref{Xloccontr=>locconsthyp}, it suffices to show that (\ref{comppreslocconst}) holds after pulling back along each point $\An\xrightarrow{y}\Y$. Denote by $P_y$ the pullback $\sh{\An}\times_{\sh{\Y}}\sh{\X}$, and let $\sh{\X_y}\xrightarrow{c}P_y$ be the comparison map. 
    As in the proof of \cref{thm:shapefibsandpreslocconst} we have to compute that the Beck-Chevalley map of the square
    \begin{equation*}
        \begin{tikzcd}
            \An_{/\X_y} \ar[d, "\pf{\sh{f_y}}"] & \An_{/\X} \ar[l] \ar[d, "\pf{\sh{f}}"] \\
            \An & \An_{/\sh{Y}} \ar[l]
        \end{tikzcd}
    \end{equation*}
    is an equivalence, which again is equivalent to $c \colon \sh{\X_y} \rightarrow P_y$ being an equivalence, concluding the proof.
\end{proof}

\subsection{Approximate fibrations in higher topos theory}

\label{subsec:approximate_fibrations}

Now we come to the definition of an approximate fibration between $\infty$-topoi.
In this section, we make extensive use of the internal higher category theory recalled in \cref{subsec:internal_higher_cats} and the subsequent sections.

\begin{definition}\label{def:approximatefibtopoi}
    Let $\X\xrightarrow{f}\Y$ be any geometric morphism. We say that $f$ is an \textit{approximate fibration} if the $\Y$-functor $\pf{F}\pb{F}$ is internally  corepresentable. When $f$ is an approximate fibration, we write $\pfs{f}(1_{\X})$ for the object of $\Y$ corepresenting the $\Y$-functor $\pf{F}\pb{F}$.
\end{definition}

Explicitly, we ask that $F_*F^* \simeq \Map_{\Omega_{\Y}}(\pfs{f}(1_{\X}),-)$ as $\Y$-functors. By the Yoneda lemma for internal higher categories, prescribing a natural transformation $\Map_{\Omega_{\Y}}(\pfs{f}(1_{\X}),-) \rightarrow F_* F^*$ is equivalent to prescribing a global section  $ 1_{\Y} \xrightarrow{s} \pf{f} \pb{f} \pfs{f}(1_{\X})$. We say that $s$ \textit{exhibits $\pfs{f}(1_{\X})$ as corepresenting object for the approximate fibration $f$} if the induced map $\Map_{\Omega_{\Y}}(\pfs{f}(1_{\X}),-) \rightarrow F_* F^*$ is an equivalence.

\begin{example}
    \label{ex:cell_like_are_approx_fib}
    Let $\X\xrightarrow{p}\Y$ be a cell-like geometric morphism. Then the equivalence $1_{\Y} \simeq \pf{p}\pb{p} 1_{\Y}$ exhibits $1_{\Y}$ as a corepresenting object for the approximate fibration $p$. More generally, if $\X \xrightarrow{p} \Y$ is cell-like and $\Y \xrightarrow{f} \Z$ an approximate fibration, then $fp$ is an approximate fibration. Indeed, in this case $\pf{F} \pf{P} \pb{P} \pb{F} \simeq \pf{F} \pb{F}$.
\end{example}

\begin{example}
    \label{ex:locally_contractible_approx}
    Every locally contractible geometric morphism $f \colon \X \rightarrow \Y$ is an approximate fibration.
    Indeed, the $\Y$-functor $\pf{F} \pb{F}$ has a left adjoint $\pfs{F} \pb{F}$, and we compute
    \[ \pf{F} \pb{F} \simeq \Map_{\Omega_{\Y}}(1_{\Y}, \pf{F} \pb{F}(-)) \simeq \Map_{\Omega_{\Y}}(\pfs{F} \pb{F}(1_{\Y}), -) \simeq \Map_{\Omega_{\Y}}(\pfs{F}(1_{\X}),-).  \]
    So, if $\pfs{f}$ denotes the left adjoint to $\pb{f}$, the corepresenting object is $\pfs{f}(1_{\X})$, and the section exhibiting it corepresenting $\pf{F} \pb{F}$ is
    \[  1_{\Y} \simeq \pf{f} 1_{\X} \xrightarrow{\mathrm{unit}} \pf{f} \pb{f} \pfs{f} 1_{\X}. \]
    This in particular applies to \'etale geometric morphisms.
\end{example}

\begin{remark}\label{rmk:propapproxfibtopoilocconstrelshape}
    Suppose that $\X\xrightarrow{f}\Y$ is an approximate fibration, and moreover assume that $f$ is a proper geometric morphism. Then the corepresenting object $\pfs{f}(1_{\X})$ for $\pf{f}\pb{f}$ is locally constant with compact values. Indeed, since $f$ is proper, the $\Y$-functor $\pf{F}\pb{F}$ preserves internally filtered colimits (\cite{Martini2023a}). Therefore, $\pfs{f}(1_{\X})$ has to be internally compact, and by \cref{prop:charinternallycpt=locconstcptvalue}, also locally constant with compact values.
\end{remark}

For proper geometric morphisms $f \colon \X \rightarrow \Y$ between locally contractible topoi, we now show that the notion of an approximate fibration agrees with the notion of a shape fibration from \cref{subsec:shape_fibrations}. The basic idea is to use the commuting diagram of topoi
\begin{equation*}
    \begin{tikzcd}
        \X \ar[r, "\eta"] \ar[d, "f"] & \An_{/\sh{\X}} \ar[d, "\sh{f}"]\\
        \Y \ar[r, "\eta"] & \An_{/\sh{\Y}}
    \end{tikzcd}
\end{equation*}
and note that the right-down composite is an approximate fibration as a locally contractible geometric morphism.
In other words, $\eta f$ is an approximate fibration, and our strategy is to use this to deduce that $f$ is an approximate fibration.
This needs some preparation, starting with the construction of a candidate for a corepresenting object $f_\sharp(1)$.

\begin{construction}
Let $f \colon \X \rightarrow \Y$ be a geometric morphism between locally contractible topoi.
We construct a morphism in $\An_{/\sh{\Y}}$ as the composite
\begin{align}\label{globseclocsys}
 1_{\sh{\Y}}\simeq\pf{\sh{f}}1_{\sh{\X}} & \xrightarrow{\text{unit}}\pf{\sh{f}}\pb{\sh{f}}\pfp{\sh{f}}1_{\sh{\X}} \\
    & \xrightarrow[\simeq]{\text{unit}} \pf{\sh{f}}\pf{\eta}\pb{\eta}\pb{\sh{f}}\pfp{\sh{f}}1_{\sh{\X}} \\
    & \xrightarrow[\simeq]{}\pf{\eta}\pf{f}\pb{f}\pb{\eta}\pfp{\sh{f}}1_{\sh{\X}}.
    \end{align}
By adjunction, this gives
\begin{equation}\label{globsectoshapef}
1_{\Y}\simeq\pb{\eta}1_{\sh{\Y}}\rightarrow\pf{f}\pb{f}\pb{\eta}\pfp{\sh{f}}1_{\sh{\X}}.
\end{equation}
Using the Yoneda lemma for $\Y$-categories (see \cite{martini2021yoneda}), the section (\ref{globsectoshapef}) gives rise to a transformation of $\Y$-functors 
\begin{equation}\label{shapecoreptransf}
    \Map_{\Omega_{\Y}}(\pb{\eta}\pfp{\sh{f}}1_{\sh{\X}},-)\rightarrow\pf{F}\pb{F}
\end{equation}
in $\Fun_{\Y}(\Omega_{\Y},\Omega_{\Y})$. 
\end{construction}

\begin{remark}\label{rmk:induced_transformation_on_shape_topoi}
    We apply the $2$-functor $\pf{\eta} \colon \Cat_{\Y} \rightarrow \Cat_{\An_{/\sh{\Y}}}$ to \ref{shapecoreptransf} to get a natural transformation
    \begin{equation}\label{pfshapecoreptransf}
    \pf{\eta}(\Map_{\Omega_{\Y}}(\pb{\eta}\pfp{\sh{f}}1_{\sh{\X}},-))\rightarrow \pf{\eta}(\pf{F}\pb{F})
    \end{equation}
    in $\Fun_{\An_{/\sh{\Y}}}(\pf{\eta} \Omega_\Y, \pf{\eta}\Omega_\Y)$.
    Viewing $\eta$ as a geometric morphism of $\An_{/\sh{\Y}}$-topoi $\boldsymbol{\eta} \colon \pf{\eta} \Omega_{\Y} \rightarrow \Omega_{\An_{/\sh{\Y}}}$, we can precompose \ref{pfshapecoreptransf} with $\pb{\boldsymbol{\eta}}$ and postcompose it with $\pf{\boldsymbol{\eta}}$. Fully faithfulness of $\pb{\boldsymbol{\eta}}$ and \cref{lemma:shomlocconstislocconst} imply that $\pf{\boldsymbol{\eta}} \circ \pf{\eta}(\Map_{\Omega_{\Y}}(\pb{\eta}\pfp{\sh{f}}1_{\sh{\X}},-)) \circ \pb{\boldsymbol{\eta}}$ agrees with the $\An_{/\sh{\Y}}$-functor $\Map_{\Omega_{\An_{/\sh{\Y}}}}(\pfp{\sh{f}}1_{\sh{\X}},-) \in \Fun_{\An_{/\sh{\Y}}}(\Omega_{\An_{/\sh{\Y}}}, \Omega_{\An_{/\sh{\Y}}})$. The composite $\pf{\boldsymbol{\eta}} \circ \pf{\eta}(\pf{F} \pb{F}) \circ \pb{\boldsymbol{\eta}}$ agrees with $\pf{\sh{F}}\pb{\sh{F}}$, by naturality and fully faithfulness of $\eta$. So we have constructed a natural transformation of $\An_{/\sh{\Y}}$-functors 
    \begin{equation}
        \label{eq:transformation_on_shape_topoi}
        \Map_{\Omega_{\An_{/\sh{\Y}}}}(\pfp{\sh{f}}1_{\sh{\X}},-) \rightarrow \pf{\sh{F}}\pb{\sh{F}}
    \end{equation}
    which we want to identify using Yoneda's lemma for $\An_{/\sh{\Y}}$-categories. Evaluating it on the identity in $\Map_{\Omega_{\An_{/\sh{\Y}}}}(\pfp{\sh{f}}1_{\sh{\X}},\pfp{\sh{f}}1_{\sh{\X}})$ we get a section of $\pf{\sh{f}} \pb{\sh{f}} \pfp{\sh{f}}1_{\sh{\X}}$. Unraveling, this section is obtained by applying \ref{pfshapecoreptransf} to the identity of $\pb{\eta} \pfp{\sh{f}}1_{\sh{\X}}$, and using the identification $\pf{\eta} \pf{f} \pb{f} \pb{\eta} \simeq \pf{\sh{f}} \pf{\eta} \pb{\eta} \pb{\sh{f}}\simeq \pf{\sh{f}} \pb{\sh{f}}$. But by construction, this is the section \ref{globseclocsys}.
\end{remark}

\begin{remark}\label{restrglobsectoshapef}
    The Yoneda lemma for $\An_{/\sh{Y}}$-categories implies that the map
    $$1_{\sh{\Y}}\simeq\pf{\sh{f}}1_{\sh{\X}} \xrightarrow{\text{unit}}\pf{\sh{f}}\pb{\sh{f}}\pfp{\sh{f}}1_{\sh{\X}}$$
    induces a natural transformation of $\An_{/\sh{Y}}$-functors
    \begin{equation}
        \label{eq:natural_transformation_on_shapes_which_is_evidently_equiv}
        \Map_{\Omega_{\An_{/\sh{\Y}}}}(\pfp{\sh{f}}1_{\sh{\X}},-)  \simeq \Map_{\Omega_{\An_{/\sh{\Y}}}}(\pfp{\sh{f}} \pb{\sh{f}}1_{\sh{\Y}},-)\rightarrow\pf{\sh{F}}\pb{\sh{F}},
    \end{equation}
    which is invertible, see \cref{ex:locally_contractible_approx}. 
    In \cref{rmk:induced_transformation_on_shape_topoi} we have seen how this natural transformation arises from \ref{shapecoreptransf}, essentially by restricting it to locally constant sheaves and postcomposing with $\pf{\eta}$. In particular, this identifies the map 
    \begin{equation*}
        \pf{\eta}\sHom{\Y}{\pb{\eta} \pfp{\sh{f}}1_{\sh{\X}}}{\pb{\eta}(-)} \rightarrow \pf{\eta}\pf{f} \pb{f} \pb{\eta}
    \end{equation*}
    arising from \ref{shapecoreptransf} by taking global sections with the map
    \begin{equation}
        \sHom{\An_{/\sh{\Y}}}{\pfp{\sh{f}}1_{\sh{\X}}}{-} \rightarrow \pf{\sh{f}} \pb{\sh{f}}
    \end{equation}
    arising from \ref{eq:natural_transformation_on_shapes_which_is_evidently_equiv}.
\end{remark}

\begin{theorem}\label{thm:approxfibandshapefib}
    Let $\X\xrightarrow{f}\Y$ be a proper geometric morphism, and suppose that $\X$ and $\Y$ are locally contractible. Then the following are equivalent.
    \begin{enumerate}[(1)]
        \item $f$ is an approximate fibration;
        \item $f$ is a shape fibration.
    \end{enumerate}
    Moreover, in this case, for each $U\in\Y$, the representing object for $f_{/U}$ is the locally constant object $\pb{\eta} \pfp{\sh{f_{/U}}}1_{\sh{\X_{/\pb{f}U}}}$.
\end{theorem}

\begin{proof}
     We start by showing that (1) implies (2). By \cref{thm:shapefibsandpreslocconst}, it suffices to show that $\pf{f}\pb{f}$ preserves locally constant objects. Let $\pfs{f}1_{\X}$ be the internal corepresenting object for $\pf{f}\pb{f}$. Since $f$ is proper, $\pfs{f}1_{\X}$ is locally constant by \cref{rmk:propapproxfibtopoilocconstrelshape}. We can then conclude by observing that $\sHom{\Y}{\pfs{f}1_{\X}}{G}\simeq \pf{f}\pb{f}G$ is locally constant whenever $G$ is, by \cref{lemma:shomlocconstislocconst}.

     We show that (2) implies (1). 
     We want to show that the transformation (\ref{shapecoreptransf}) is a natural isomorphism of $\Y$-functors. Notice that it suffices to show that (\ref{shapecoreptransf}) is invertible after restricting to the $\Y$-category of locally constant objects. Indeed, since $f$ is proper and (2) holds, by \cref{lemma:approxfib=>shapequasifib} $f$ is a shape quasi-fibration. Moreover, for any point $y$ in $\Y$, the shape of $\X_y$ is compact by \cref{lem:quasi-fibbtweenconstshap=>constshapefib}. Therefore,  $\pb{\eta}\pfp{\sh{f}}1_\X$ is locally constant with compact values by \cref{rmk:cptvalues&cptstalks}, and therefore internally compact by \cref{prop:charinternallycpt=locconstcptvalue}. Thus, we deduce that the $\Y$-functor $\Map_{\Omega_{\Y}}(\pb{\eta}\pfp{\sh{f}}1_{\sh{\X}},-)$ preserves internally filtered colimits. As a consequence, we find that (\ref{shapecoreptransf}) is invertible if and only if its restriction to internally compact objects is invertible. Our claim then follows from \cref{prop:charinternallycpt=locconstcptvalue}.

     We now observe that one may further reduce to proving that the restriction of (\ref{shapecoreptransf}) to locally constant objects is invertible after passing to global sections. In other words, we now show that it suffices to prove that the natural transformation 
     \begin{equation}\label{globsecshapecoreptransf}
     	\sHom{\Y}{\pb{\eta}\pfp{\sh{f}}1_{\sh{\X}}}{\pb{\eta}(-)}\rightarrow\pf{f}\pb{f}\pb{\eta}
     \end{equation}
     induced by (\ref{shapecoreptransf}) is invertible. Let $U\in\Y$ be any object, and denote by $\pi_U:\Y_{/U}\rightarrow\Y$ the associated \'etale geometric morphism. Proving invertibility of (\ref{shapecoreptransf}) amounts to checking that, for each $U\in\Y$, the map
     \begin{equation}\label{shapecoreptransfU}
     	\sHom{\Y_{/U}}{\pb{\pi_U}\pb{\eta}\pfp{\sh{f}}1_{\sh{\X}}}{\pb{\eta}(-)}\rightarrow\pf{{f_{/U}}}\pb{{f_{/U}}}\pb{\eta}
     \end{equation}
     is invertible. Observe that by assumption (2) and naturality of the functor $\sh{(-)}$ (\cref{rmk:etanatural}), base change for left Kan extension provides an isomorphism \[\pb{\pi_U}\pb{\eta}\pfp{\sh{f}}1_{\sh{\X}} \simeq \pb{\eta} \pb{\pi_\sh{\Y_{/U}}} \sh{f_{/U}}_! 1_{\sh{\X}_{/\pb{f}U}} \simeq\pb{\eta}\pfp{\sh{{f_{/U}}}}1_{\sh{\X_{/\pb{f}U}}}.\] Therefore we see that (\ref{shapecoreptransfU}) is invertible if and only if 
     $$\sHom{\Y_{/U}}{\pb{\eta}\pfp{\sh{{f_{/U}}}}1_{\sh{\X_{/\pb{f}U}}}}{\pb{\eta}(-)}\rightarrow\pf{{f_{/U}}}\pb{{f_{/U}}}\pb{\eta}$$
     is invertible. By \cref{rmk:shape(quasi)fibstableunderetalebasechange}, the geometric morphism $f_{/U}$ is a shape fibration whenever $f$ is, and so our claim is proven.

     We now prove that (\ref{globsecshapecoreptransf}) is invertible if $f$ is a shape fibration. By naturality of counits, we have a commuting square 
    $$
    \begin{tikzcd}
\pb{\eta}\pf{\eta}\sHom{\Y}{\pb{\eta}\pfp{\sh{f}}1_{\sh{\X}}}{\pb{\eta}(-)} \arrow[d] \arrow[rr] &  & \pb{\eta}\pf{\eta}\pf{f}\pb{f}\pb{\eta} \arrow[d] \\
\sHom{\Y}{\pb{\eta}\pfp{\sh{f}}1_{\sh{\X}}}{\pb{\eta}(-)} \arrow[rr]                             &  & \pf{f}\pb{f}\pb{\eta}                   
\end{tikzcd}
    $$
    where the lower horizontal arrow is the transformation (\ref{shapecoreptransf}). The left vertical arrow is invertible by \cref{lemma:shomlocconstislocconst}, the upper horizontal arrow is invertible by \cref{restrglobsectoshapef}, and the right vertical arrow is invertible by \cref{thm:shapefibsandpreslocconst}, using assumption (2). Therefore, we obtain that the lower horizontal arrow is invertible, and so our proof is concluded.
\end{proof}

\begin{corollary} 
    \label{cor:approximate_fibrations_characterization}
    Let $\X\xrightarrow{f}\Y$ be a proper geometric morphism, and assume that $\X$ and $\Y$ are locally contractible $\infty$-topoi. Assume moreover that $\Y\hyp$ has enough points. Then the following conditions are equivalent.
    \begin{enumerate}[(1)]
        \item $f$ is an approximate fibration;
        \item $\pf{f}\pb{f}$ preserves locally constant objects;
        \item $f$ is a shape fibration;
        \item $f$ is a shape quasi-fibration.
    \end{enumerate}
\end{corollary}

\begin{proof}
    This is an immediate consequence of \cref{thm:approxfibandshapefib} and \cref{thm:quasifibsandpreslocconst}.
\end{proof}

\begin{remark}\label{rmk:cell-likefromapproxfibloccontr}
    Note that \cref{cor:cell_like_characterization} is a consequence of \cref{thm:approxfibandshapefib} and \cref{cor:approximate_fibrations_characterization}
    in the case where $\X \xrightarrow{f} \Y$ is a map of locally contractible $\infty$-topoi. Indeed, if $f$ is cell-like then it is an approximate fibration with corepresenting object $1_\Y$ by \cref{ex:cell_like_are_approx_fib}. Hence, it is a shape fibration, and for each $U  \in \Y$ we have $\pfp{\sh{f}} 1_{\sh{\X_{/\pb{f}U}}} \simeq 1_{\sh{\Y_{/U}}} $. This implies that $\sh{f_{/U}} \colon \sh{\X_{/f^* U}} \rightarrow \sh{Y_{/U}}$ is an equivalence. Conversely, hereditary shape equivalences are clearly shape fibrations, and the same computation forces the internal corepresenting object to be $1_{\Y}$ in that case, so that $f_* f^*$ is the identity on $\Y$. If $\Y^{\hyp}$ has enough points, then under the equivalence of (3) and (4) in \cref{cor:approximate_fibrations_characterization}, hereditary shape equivalences correspond to pointwise shape equivalences.
    However, \cref{cor:cell_like_characterization} is strictly more general as it also applies to topoi which are not locally contractible, e.g. sheaf topoi of profinite topological spaces.
\end{remark}

\subsection{Approximate fibrations in topology}

\label{sec:approximate_fibrations}

We now recall the notion of approximate fibrations in topology, introduced by Coram and Duvall in \cite{CoramDuvall}, and subsequently put it in the context of higher topos theory.

For terminology, suppose $Y$ is a topological space and $\epsilon = \{ U_\alpha \}$ is an open cover of $Y$. We say that two maps $f, g \colon Z \rightarrow Y$ are \textit{$\epsilon$-close} if for each $z \in Z$ there is an $\alpha$ with both $f(z)$ and $g(z)$ lying in $U_\alpha$. The suggestive example to keep in mind is if $Y$ is a metric space and $\epsilon$ is a positive real number, we can associate to it the open cover of balls of radius $\epsilon$.

\begin{definition}[\cite{CoramDuvall}]
    Let $p \colon X \rightarrow Y$ be a proper map of ANRs. Then $p$ is called an approximate fibration if for every square
    \begin{equation*}
        \begin{tikzcd}
            Z \times \{0\} \ar[r, "l"] \ar[d] & X \ar[d, "p"] \\
            Z \times [0,1] \ar[r, "H"] & Y
        \end{tikzcd}
    \end{equation*}
    and every open cover $\epsilon = \{ U_\alpha \}$ of $Y$, there exists a map $H' \colon Z \times [0,1] \rightarrow Y$ which is $\epsilon$-close to $H$, agrees with $H$ on $Z \times \{0\}$, and admits a diagonal lift $L \colon Z \times [0,1] \rightarrow X$ with $pL = H'$ which agrees with $l$ on $Z \times \{0\}$.
\end{definition}

We now give a reformulation of the notion of an approximate fibration, which is due to \cite{HughesTaylorWilliams}. In the following, we write $\Pi_{\infty}$ to denote the weak homotopy type of a topological space.

\begin{theorem}\label{thm:hughestaylorwilliams}
    Let $f \colon X \rightarrow Y$ be a proper map of ANRs. Then the following are equivalent.
    \begin{enumerate}[(1)]
        \item The map $f$ is an approximate fibration.
        \item For each open subset $U \subset Y$, the square
        \begin{equation*}
            \begin{tikzcd}
                \Pi_{\infty}(f^{-1}(U)) \ar[r] \ar[d] & \Pi_{\infty}(X) \ar[d,"f"] \\
                \Pi_{\infty}(U) \ar[r] & \Pi_{\infty}(Y)
            \end{tikzcd}
        \end{equation*}
        is a pullback in $\An$.
    \end{enumerate}
\end{theorem}

\begin{proof}
    Recall that a square
    \begin{equation*}
            \begin{tikzcd}
                E' \ar[r] \ar[d, "q"] & E \ar[d,"p"] \\
                B' \ar[r, "s"] & B
            \end{tikzcd}
        \end{equation*}
    in $\An$ is a pullback if and only if, for each $x\in B'$, the natural map from the fiber of $q$ at $x$ to the fiber of $p$ at $s(x)$ is invertible. Therefore, we see that the condition appearing in \cite[Theorem 12.15]{HughesTaylorWilliams} is equivalent to condition (2) in our statement. Hence, the proof is concluded by combining \cite[Theorem 12.13]{HughesTaylorWilliams} and \cite[Theorem 12.15]{HughesTaylorWilliams}.
\end{proof}

\begin{theorem}
    \label{thm:characterization_of_approximate_fibrations}
    Let $f \colon X \rightarrow Y$ be a proper map of ANRs. Then the following are equivalent.
    \begin{enumerate}[(1)]
        \item The map $f$ is an approximate fibration.
        \item The geometric morphism $f:\Shsp{X}\rightarrow\Shsp{Y}$ is an approximate fibration.
    \end{enumerate}
\end{theorem}

\begin{proof}
    By \cref{thm:approxfibandshapefib}, it suffices to show that (1) holds if and only if $f:\Shsp{X}\rightarrow\Shsp{Y}$ is a shape fibration. The latter condition is equivalent to (2) in \cref{thm:hughestaylorwilliams}. Indeed, this is seen by combining \cref{cor:shapefibontopspace} with the fact that ANRs are of singular shape, proven in \cref{ex:ANRs}. The proof is thus finished by appealing to \cref{thm:hughestaylorwilliams}.
\end{proof}

\appendix

\section{A formula for stalks}

The purpose of this appendix is to give a formula for the shape of fibers of proper geometric morphisms of $\infty$-topoi. As a consequence, we get a description of the shape of fibers of proper, separated maps of sober topological spaces.

Let $C$ be any small $\infty$-category. Let $f \colon C\rightarrow\An$ be any functor, and write $\pi_f  \colon E_f \rightarrow C$ for the left fibration classified by $f$. Denote by $\pfp{f} \colon \PSh{C}\rightarrow\An$ the extension by colimits of $f$. The following lemma provides a convenient formula to describe $\pfp{f}$ in terms of $\pi_f$.

\begin{lemma}\label{lemma:extbycolimformula}
    The functor $\pfp{f}$ is naturally equivalent to the composition
    $$\PSh{C}\xrightarrow{\pb{{\pi_f}}}\PSh{E_f}\xrightarrow{\varinjlim}\An.$$
    More explicitly, for any $F\in\PSh{C}$, we have an isomorphism
    $$\pfp{f}F\simeq\varinjlim\limits_{e\in E_f}F(\pi_f(e))$$
    which is natural in $F$.
\end{lemma}

\begin{proof}
This is a special case of \cite[Lemma 2.3.4.3]{martini2025presentabilitytopoiinternalhigher}. \qedhere%
\end{proof}

Let $\X$ be an $\infty$-topos, and let $\An\xrightarrow{x}\X$ be any point. Suppose that $C$ is a small $\infty$-category and $\PSh{C}\xrightarrow{L}\X$ is a left exact localization with fully faithful right adjoint $\X\xrightarrow{i}\PSh{C}$. Let $C\xrightarrow{f_x}\An$ be the composite
$C\xrightarrow{y}\PSh{C}\xrightarrow{L}\X\xrightarrow{\pb{x}}\An$. Since $x^*L$ commutes with colimits, it is the unique colimit preserving extension of $f_x$.

\begin{definition}\label{def:C-neighbourhoods}
    We define the $\infty$-category of \textit{$C$-neighbourhoods of the point $x$} to be the domain of the left fibration classifying the functor $C\xrightarrow{f_x}\An$. We denote the latter by $N_C(x)\xrightarrow{\pi_x}C$.
\end{definition}

\begin{remark}\label{rmk:N(x)isfilt}
    Notice that the $\infty$-category $N_C(x)$ is filtered. Indeed, the extension by colimits $(f_x)_! \colon \PSh{C} \rightarrow \An$ of the functor $f_x$ coincides with the composition $\pb{x}L$. Since the latter is left exact, our claim follows from \cite[Proposition 3.2.8.1]{martini2025presentabilitytopoiinternalhigher}.
\end{remark}

\begin{remark}
    Notice that there is an isomorphism
    $$\pb{x}(-)\simeq\Hom{\Top}{\An}{\An_{/\pb{x}(-)}}$$
    in $\Fun(\X,\An)$ (this is a very special case of \cite[Corollary 6.3.5.6]{lurie2009higher}). Moreover, by \cite[Remark 6.3.5.8]{lurie2009higher}, there is a pullback square in $\Fun(\X,\Top)$ of the form
    $$
    \begin{tikzcd}
\An_{/\pb{x}(-)} \arrow[d] \arrow[r] & \X_{/(-)} \arrow[d] \\
\An \arrow[r, "x"]                    & \X.                  
\end{tikzcd}
    $$
    As a consequence, we deduce that objects in $N_C(x)$ can be informally described as pairs $(U, x_U)$, where $U$ is an object in $C$ and $x_U$ is the datum of a point $\An\xrightarrow{x_U}\X_{/L(y(U))}$ and a commuting triangle 
    $$
    \begin{tikzcd}
                                     & \X_{/L(y(U))} \arrow[d] \\
\An \arrow[ru, "x_U"] \arrow[r, "x"] & \X.                 
\end{tikzcd}
    $$
    This justifies our choice for the name given to $N_C(x)$.
\end{remark}

We are now ready to prove our desired explicit formula for $\pb{x}$.

\begin{proposition}\label{prop:formulastalks}
    Let $F$ be any object in $\X$. Then we have an isomorphism
    $$\pb{x}F\simeq\varinjlim\limits_{(U,x_U)\in N_C(x)}\Hom{\X}{L(U)}{F}$$
    which is natural in $F$.
\end{proposition}

\begin{proof}
    We have
    \begin{align*}
        \pb{x}F&\simeq\pb{x}Li(F) \\
        &\simeq\varinjlim\limits_{(U,x_U)\in N_C(x)}\Hom{\PSh{C}}{y(U)}{i(F)} \\
        &\simeq \varinjlim\limits_{(U,x_U)\in N_C(x)}\Hom{\X}{L(y(U))}{F}
    \end{align*}
    where the first isomorphism follows since $i$ is fully faithful, the second by \cref{lemma:extbycolimformula} and the third by adjunction.
\end{proof}

\begin{corollary}\label{cor:shapeoffibpropermorph}
    Let $\Y$ be any $\infty$-topos, and $C$ a small $\infty$-category with a left exact localization $\PSh{C}\xrightarrow{L}\Y$. Let $\X\xrightarrow{f}\Y$ be a proper geometric morphism, and let $\An\xrightarrow{y}\Y$ be any point. Then there is an isomorphism
    $$\sh{\X_y}\simeq\varprojlim\limits_{(U,y_U)\in N_C(y)\op}\sh{\X_{/\pb{f}L(y(U))}}$$
    in $\Pro(\An)$.
\end{corollary}

\begin{proof}
    Let $\X_y\xrightarrow{a}\An$ and $\X\xrightarrow{b}\An$ be the unique geometric morphisms. Since $f$ is proper, we have a natural isomorphism
    $$\pf{a}\pb{a}\simeq\pb{y}\pf{f}\pb{b}.$$
    The proof is then concluded by observing that the left-hand side is corepresented in $\Pro(\An)$ by $\sh{\X_y}$, while the right-hand side is corepresented by $\varprojlim\limits_{(U,y_U)\in N_C(y)\op}\sh{\X_{/\pb{f}L(y(U))}}$. The last claim is a direct consequence of \cref{prop:formulastalks} and \cref{rmk:N(x)isfilt}.
\end{proof}

\begin{example}
    Let $f \colon X \rightarrow Y$ be a map of topological spaces, and let $y \in Y$ be a point. We can write $\Shsp{Y}$ as a left-exact localization of $\PSh{\mathcal{O}(Y)}$, the category of presheaves on the category of open subsets of $Y$. The category $N_{\mathcal{O}(Y)}(y)$ can be described by considering the composition
    \[ \mathcal{O}(Y) \rightarrow \Shsp{Y} \xrightarrow{y^*} \An \]
    which sends $U \subset Y$ to $1 \in \An$ if $y \in U$ and to $\emptyset$ otherwise. The associated left fibration is the inclusion of the full subcategory $\mathcal{O}(Y)_y \subset \mathcal{O}(Y)$ on such open subsets containing the point $y$.
\end{example}

\begin{corollary}\label{cor:appendixshapefibers}
    Let $f\colon X\to Y$ be a proper and separated map of sober topological spaces, and let $y\in Y$ be any point. Then we have an equivalence
    \[ \sh{X_y} \simeq \varprojlim\limits_{y \in U \subset Y}\sh{f^{-1}(U)} \]
    in $\Pro(\An)$.
\end{corollary}

\begin{proof}
    Thanks to \cref{cor:shapeoffibpropermorph}, it suffices to show that $f$ induces a proper geometric morphism, and that $\Sh{X_y}{\An}$ agrees with the corresponding fiber taken in the $\infty$-category of $\infty$-topoi. The assertions are proven respectively in \cite[Theorem 3.5.1]{Martini2023a} and \cite[Corollary 3.5.4]{Martini2023a} (see also \cite[Lemma 2.1, Corollary 3.6]{Johnstone1981GleasonCoverII}).
\end{proof}

\newpage
\nocite{*}
\bibliographystyle{alpha}
\bibliography{approx}
\end{document}